\documentclass[smallextended,runningheads]{svjour3}
\usepackage{amsmath,amssymb,mathtools,mathabx,graphicx}
\journalname{Accepted}
\newcommand*{\id}{{\mathrm{id}}}

\newcommand*{\R}{{\mathbb R}}                                                 
\newcommand*{\C}{{\mathbb{C}}}

\newcommand*{\Hb}{{\mathbb H}}

\newcommand*{\Pb}{{\mathbb P}}                                                
\newcommand*{\Qb}{{\mathbb Q}}

\newcommand*{\pa}{{\partial}}

\newcommand*{\rd}{{\mathrm d}}

\newcommand*{\Gr}{{\mathrm{Gr}}}

\smartqed

\begin{document}
\title{Partial differential systems with nonlocal nonlinearities: Generation and solutions}
\author{Margaret Beck \and Anastasia Doikou \and
Simon~J.A.~Malham \and Ioannis Stylianidis}
\titlerunning{Partial differential systems with nonlocal nonlinearities}   
\authorrunning{Beck \and Doikou \and Malham \and Stylianidis}      

\institute{Margaret Beck \at Department of Mathematics and Statistics,
Boston University, Boston MA~02215, USA \email{mabeck@bu.edu}
\and Anastasia Doikou, Simon~J.A.~Malham and Ioannis Stylianidis
\at Maxwell Institute for Mathematical Sciences,        
and School of Mathematical and Computer Sciences,   
Heriot-Watt University, Edinburgh EH14 4AS, UK 
\email{A.Doikou@hw.ac.uk, S.J.A.Malham@hw.ac.uk, is11@hw.ac.uk}}

\date{30th January 2018}   

\maketitle

\begin{abstract}
We develop a method for generating solutions to large classes of evolutionary
partial differential systems with nonlocal nonlinearities. 
For arbitrary initial data, the solutions are generated from the corresponding 
linearized equations. The key is a Fredholm integral equation relating the 
linearized flow to an auxiliary linear flow. It is analogous to the 
Marchenko integral equation in integrable systems. We show explicitly how
this can be achieved through several examples including reaction-diffusion 
systems with nonlocal quadratic nonlinearities and the nonlinear Schr\"odinger
equation with a nonlocal cubic nonlinear-ity. In each case we demonstrate our
approach with numerical simulations. We discuss the effectiveness of our approach
and how it might be extended. 
\end{abstract}

\section{Introduction}\label{sec:intro}
Our concern is the generation of solutions to nonlinear partial
differential equations. In particular, as is natural, to develop
methods that generate such solutions from solutions to the corresponding
linearized equations. 
Herein we do not restrict ourselves to soliton equations, nor indeed 
to integrable systems. We do not demand nor require the existence of a Lax pair.
However our approach herein as it stands at this time, only applies to classes 
of partial differential systems with nonlocal nonlinearities.
Naturally we seek to extend it to more general systems and we discuss how this
might be achieved in our conclusions. However let us return to what we have 
achieved thus far and intend to achieve herein. 
In Beck, Doikou, Malham and Stylianidis~\cite{BDMS} we demonstrated the 
approach we developed indeed works for large classes of 
scalar partial differential equations with quadratic nonlocal nonlinearities. 
For example we demonstrated,
for general smooth initial data $g_0=g_0(x,y)$ with $x,y\in\R$ and 
some time $T>0$ of existence, how to construct solutions
$g\in C^\infty\bigl([0,T];C^{\infty}(\R^2;\R)\cap L^2(\R^2;\R)\bigr)$ 
to partial differential equations of the form 
\begin{equation*}
\pa_tg(x,y;t)=d(\pa_x)g(x,y;t)-\int_\R g(x,z;t)\,b(\pa_z)g(z,y;t)\,\rd z.
\end{equation*}
In this equation, $d=d(\pa_x)$ is a polynomial function of the partial differential
operator $\pa_x$ with constant coefficients, while $b$ is either a polynomial
function $b=b(\pa_x)$ of $\pa_x$ with constant coefficients, or it is a smooth
bounded function $b=b(x)$ of $x$. Thus the linear term $d(\pa_x)\,g(x,y;t)$ is quite general, 
while the quadratic nonlinear term, whilst also quite general, has the nonlocal form shown. 
Hereafter for convenience we denote this nonlocal product by `$\star$', 
defined for any two functions $g,g^\prime\in L^2(\R^2;\R)$ by 
\begin{equation*}
\bigl(g\star g^\prime\bigr)(x,y)\coloneqq\int_\R g(x,z)\,g^\prime(z,y)\,\rd z.
\end{equation*}
Hence for example the nonlocal nonlinear term above can be
expressed as $\bigl(g\star (bg)\bigr)(x,y;t)$.

In this paper we extend our method in two directions. First we extend it to classes of 
systems of partial differential equations with quadratic nonlocal nonlinearities. For
example we demonstrate, for general smooth initial data $u_0=u_0(x,y)$ and
$v_0=v_0(x,y)$ with $x,y\in\R$ and some time $T>0$, how to construct 
solutions $u,v\in C^\infty\bigl([0,T];C^{\infty}(\R^2;\R)\cap L^2(\R^2;\R)\bigr)$
to partial differential systems with quadratic nonlocal
nonlinearities of the form
\begin{align*}
\pa_tu&=d_{11}(\pa_1)u+d_{12}(\pa_1)v-u\star(b_{11}u)-u\star(b_{12}v)-v\star(b_{12}u)-v\star(b_{11}v),\\
\pa_tv&=d_{11}(\pa_1)v+d_{12}(\pa_1)u-u\star(b_{11}v)-u\star(b_{12}u)-v\star(b_{12}v)-v\star(b_{11}u).
\end{align*}
In this formulation the operators $d_{11}=d_{11}(\pa_1)$, 
$d_{12}=d_{12}(\pa_1)$ are polynomials of $\pa_1$ analogous to the 
operator $d$ above, the operation $\star$ is as defined above 
and $b_{11}$ and $b_{12}$ are analogous functions to the function $b$ defined above.
In the special case that $d_{11}$ and $d_{22}$ are both constant multiples of $\pa_1^2$ 
and $b_{11}$ and $b_{12}$ are scalar constants, then the system of equations for
$u$ and $v$ above represent a system of reaction-diffusion equations
with nonlocal nonlinear reaction/interaction terms.

Second, with a slight modification, we extend our approach to classes of 
partial differential equations with cubic and higher odd degree 
nonlocal nonlinearities. In particular, for general smooth $\C$-valued
initial data $g_0=g_0(x,y)$ with $x,y\in\R$ and some time $T>0$,
we demonstrate how to construct solutions 
$g\in C^\infty\bigl([0,T];C^{\infty}(\R^2;\C)\cap L^2(\R^2;\C)\bigr)$ 
to nonlocal nonlinear partial differential equations of the form
($\mathrm{i}=\sqrt{-1}$),
\begin{equation*}
\mathrm{i}\,\pa_tg=d(\pa_1)g+g\star f^\star(g\star g^\dag).
\end{equation*}
Here with a slight abuse of notation, we suppose
\begin{equation*}
(g\star g^\dag)(x,y)\coloneqq
\int_{\R}g(x,z)\,g^*(y,z)\,\rd z,
\end{equation*}
where $g^*$ denotes the complex conjugate of $g$. Our method works
for any choice of $d$ of the form $d=\mathrm{i}h(\pa_1)$, where $h$ is 
any constant coefficient polynomial with only even degree terms of its argument. 
Further, it works for any function $f^\star$ with a power series representation 
with infinite radius of convergence and real coefficients $\alpha_m$ of the form
\begin{equation*}
f^\star(c)=\mathrm{i}\sum_{m\geqslant0}\alpha_m c^{\star m}.
\end{equation*}
The expression $c^{\star m}$ represents the $m$-fold $\star$ product 
of $c\in L^2(\R^2;\C)$.

Our method is based on the development of Grassmannian flows from 
linear subspace flows as follows; see Beck \textit{et al.\/}~\cite{BDMS}.
Formally, suppose that $Q=Q(t)$ and $P=P(t)$ are linear operators satisfying the following 
linear system of evolution equations in time $t$, 
\begin{equation*}
\pa_tQ=AQ+BP
\qquad\text{and}\qquad  
\pa_tP=CQ+DP.
\end{equation*}
We assume that $A$ and $C$ are bounded linear operators, while $B$ and $D$ 
may be bounded or unbounded operators. Throughout their time interval of 
existence say on $[0,T]$ with $T>0$, we suppose $Q-\id$ and $P$ to be 
compact operators, indeed Hilbert--Schmidt operators. Thus $Q$ itself 
is a Fredholm operator. If $B$ and $D$ are unbounded operators we suppose
$Q-\id$ and $P$ to lie in a suitable subset of the class of Hilbert--Schmidt
operators characterised by their domains. We now posit a relation between 
$P=P(t)$ and $Q=Q(t)$ mediated through a compact Hilbert--Schmidt operator $G=G(t)$ 
as follows,
\begin{equation*}
P=G\,Q.
\end{equation*}
Suppose we now differentiate this relation with respect to time using the product rule
and insert the evolution equations for $Q=Q(t)$ and $P=P(t)$ above. If we then equivalence 
by the Fredholm operator $Q=Q(t)$, i.e.\/ post-compose by $Q^{-1}=Q^{-1}(t)$ 
on the time interval on which it exists, we obtain the following 
Riccati evolution equation for $G=G(t)$,
\begin{equation*}
\pa_tG=C+D\,G-G\,(A+B\,G).
\end{equation*}
This demonstrates how certain classes of quadratically nonlinear 
operator-valued evolution equations, i.e.\/ the equation for $G=G(t)$ above,
can be generated from a coupled pair of linear operator-valued equations,
i.e.\/ the equations for $Q=Q(t)$ and $P=P(t)$ above. We think of the prescription 
just given as the ``abstract'' setting in which $Q=Q(t)$, $P=P(t)$ and $G=G(t)$
are operators of the classes indicated. Note that often we will take $A=C=O$
and the equations for $Q=Q(t)$ and $P=P(t)$ above are $\pa_tQ=BP$ and $\pa_tP=DP$.
In this case, once we have solved the evolution equation for $P=P(t)$, we can
then solve the equation for $Q=Q(t)$.

We can generate cubic and higher odd degree classes of nonlinear 
operator-valued evolution equations analogous to that for $G=G(t)$ above
by slightly modifying the procedure we outlined. Again, formally,
suppose that $Q=Q(t)$ and $P=P(t)$ are linear operators satisfying the following 
linear system of evolution equations in time $t$, 
\begin{equation*}
\pa_tQ=f(PP^\dag)\,Q
\qquad\text{and}\qquad  
\pa_tP=DP,
\end{equation*}
where $P^\dag=P^\dag(t)$ denotes the operator adjoint to $P=P(t)$ 
and $f$ is a function with a power series expansion with infinite radius of convergence. 
The operator $D$ may be a bounded or unbounded operator.
In addition we require that $Q=Q(t)$ satisfies the constraint $QQ^\dag=\id$ while 
it exists. Indeed as above, throughout their time interval of 
existence say on $[0,T]$ with $T>0$, we suppose $Q-\id$ and $P$ to be 
Hilbert--Schmidt operators. If $D$ is unbounded then we suppose $P$ lies 
in a suitable subset of the class of Hilbert--Schmidt operators characterised by its domain.
We can think of the equations above as corresponding to the previous
set of equations for $Q=Q(t)$ and $P=P(t)$ in the paragraph above with
the choice $B=C=O$ and $A=f(PP^\dag)$. We emphasize however, once we have 
solved the evolution equation for $P=P(t)$, the evolution equation for $Q=Q(t)$ is linear. 
We posit the same linear relation $P=G\,Q$ between $P=P(t)$ and $Q=Q(t)$ as before, mediated
through a compact Hilbert--Schmidt operator $G=G(t)$.
Then a direct analogous calculation to that above, differentiating this relation with respect 
to time and so forth, reveals that $G=G(t)$ satisfies the evolution equation
\begin{equation*}
\pa_tG=D\,G-G\,f(GG^\dag).
\end{equation*}
The requirement that $Q=Q(t)$ must satisfy the constraint $QQ^\dag=\id$
induces the requirement that $f^\dag=-f$. Hence again, we can generate 
certain classes of cubic and higher odd degree nonlinear 
operator-valued evolution equations, like that for $G=G(t)$ just above,
by first solving the operator-valued linear evolution equation for $P=P(t)$
and then solving the operator-valued linear evolution equation for $Q=Q(t)$.
To summarize, we observe that in both procedures above, 
there were three essential components as follows, a linear:
\begin{enumerate}
\item Base equation: $\pa_tP=DP$;
\item Auxiliary equation: $\pa_tQ=BP$ or $\pa_tQ=f(PP^\dag)\,Q$;
\item Riccati relation: $P=G\,Q$.
\end{enumerate}
We now make an important observation and ask two crucial questions. 
First, we observe that solving each of the three linear
equations above in turn actually generates solutions $G=G(t)$ to the classes
of operator-valued nonlinear evolution equations shown above. Second, 
in the appropriate context, can we interpret the operator-valued nonlinear 
evolution equations above as nonlinear partial differential equations? 
Third, if so, what classes of nonlinear partial differential equations
fit into this context and can be solved in this way? In other words,
can we solve the inverse problem: given a nonlinear partial differential
equation, can we fit it into the context above (or an analogous context) 
and solve it for arbitrary initial data by solving the corresponding
three linear equations above in turn? 

Briefly and formally, keeping technical details to a minimum for the moment,
a simple example that addresses these issues, answers these questions 
positively and outlines our proposed procedure is as follows. Suppose 
$\mathbb Q$ is a closed linear subspace of $L^2(\R;\R^2)$ and that
$\mathbb P$ is the complementary subspace to $\mathbb Q$ in the 
direct sum decomposition $L^2(\R;\R^2)=\mathbb Q\oplus\mathbb P$.
Suppose for each $t\in[0,T]$ for some $T>0$ that $Q=Q(t)$ is a Fredholm operator 
from $\mathbb Q$ to $\mathbb Q$ of the form $Q=\id+Q^\prime$, and that
$Q^\prime=Q^\prime(t)$ is a Hilbert--Schmidt operator. Further we assume 
$P(t)\colon\mathbb Q\to\mathbb P$ is a Hilbert--Schmidt operator for $t\in[0,T]$.
Technically, as mentioned above, we require $Q^\prime$ and $P$ to exist in 
appropriate subspaces of the class of Hilbert--Schmidt operators. However we 
suppress this fact for now to maintain clarity and brevity (explicit details
are given in the following sections). With this context while they exist, 
$Q^\prime=Q^\prime(t)$ and $P=P(t)$ can both be represented 
by integral kernels $q^\prime=q^\prime(x,y;t)$ and $p=p(x,y;t)$, respectively, where
$x,y\in\R$ and $t\in[0,T]$. Suppose that $D=\pa_x^2$ and $B=1$ so that
the base and auxiliary equations have the form
\begin{equation*}
\pa_tp(x,y;t)=\pa_x^2p(x,y;t)\qquad\text{and}\qquad\pa_tq^\prime(x,y;t)=p(x,y;t).
\end{equation*}
The linear Riccati relation in this context takes the form of the linear Fredholm equation
\begin{equation*}
p(x,y;t)=g(x,y;t)+\int_\R g(x,z;t)\,q^\prime(z,y;t)\,\rd z.
\end{equation*}
We can express this more succinctly as $p=g+g\star q^\prime$ or $p=g\star(\delta+q^\prime)$, where
$\delta$ is the identity operator with respect to the $\star$ product. 
As described above in the ``abstract'' operator-valued setting, we can 
differentiate the relation $p=g\star(\delta+q^\prime)$ with respect to time using
the product rule and insert the base and linear equations $\pa_tp=\pa_1^2p$ and
$\pa_t q^\prime=p$ to obtain the following
\begin{align*}
(\pa_t g)\star(\delta+q^\prime)
&=\pa_tp-g\star\pa_tq^\prime\\
&=(\pa_1^2g)\star(\delta+q^\prime)-g\star(g\star(\delta+q^\prime))\\
&=(\pa_1^2g-g\star g)\star(\delta+q^\prime).
\end{align*}
In the last step we utilized the associativity property 
$g\star(g\star q^\prime)=(g\star g)\star q^\prime$ which is equivalent to the relabelling 
$\int_\R g(x,z;t)\int_\R g(z,\zeta;t)\,q^\prime(\zeta,y;t)\,\rd\zeta\,\rd z
=\int_\R\int_\R g(x,\zeta;t)\,g(\zeta,z;t)\,\rd\zeta\,q^\prime(z,y;t)\,\rd z$.
We now equivalence by $Q=Q(t)$, i.e.\/ post-compose by $\tilde Q\coloneqq Q^{-1}$. This
is equivalent to ``multiplying'' the equation above by $\star(\delta+\tilde q^\prime)$ 
where $(\delta+q^\prime)\star(\delta+\tilde q^\prime)=\delta$ and $\tilde q^\prime$ is the integral
kernel associated with $\tilde Q-\id$. We thus observe that $g=g(x,y;t)$
necessarily satisfies the nonlocal nonlinear partial differential equation
\begin{equation*}
\pa_t g=\pa_1^2g-g\star g
\end{equation*}
or more explicitly
\begin{equation*}
\pa_t g(x,y;t)=\pa_x^2g(x,y;t)-\int_\R g(x,z;t)\,g(z,y;t)\,\rd z.
\end{equation*}
Further now suppose, given initial data $g(x,y;0)=g_0(x,y)$ we wish to
solve this nonlocal nonlinear partial differential equation. We observe 
that we can explicitly solve, in closed form via Fourier transform, 
for $p=p(x,y;t)$ and then $q^\prime=q^\prime(x,y;t)$. 
We take $q^\prime(x,y;0)=0$ and $p(x,y;0)=g_0(x,y)$. 
This choice is consistent with the Riccati relation evaluated at time $t=0$.
Then the solution of the Riccati relation by iteration or other means, and in some
cases explicitly, generates the solution $g=g(x,y;t)$ to the nonlocal 
nonlinear partial differential equation above corresponding to the initial data $g_0$.
We have thus now seen the ``abstract'' setting and the connection to 
nonlocal nonlinear partial differential equations and their solution,
and thus started to lay the foundations to validating our claims at the 
very beginning of this introduction.

The approach we have outlined above, for us, has its roots in the 
series of papers in numerical spectral theory in which Riccati equations
were derived and solved in order to resolve numerical difficulties
associated with linear spectral problems. These difficulties were associated
with different exponential growth rates in the far-field. See for example 
Ledoux, Malham and Th\"ummler~\cite{LMT}, Ledoux, Malham, Niesen and Th\"ummler~\cite{LMNT}, 
Karambal and Malham~\cite{KM} and Beck and Malham~\cite{BM} for more details
of the use of Riccati equations and Grassmann flows to help numerically
evaluate the pure-point spectra of linear elliptic operators.  
In Beck \textit{et al.\/ }~\cite{BDMS} we turned the question around and 
asked whether the Riccati equations, which in infinite dimensions represent
nonlinear partial differential equations, could be solved by the reverse
process.
 
The notion that integrable nonlinear partial differential equations can 
be generated from solutions to the corresponding linearized equation
and a linear integral equation, namely the Gel'fand--Levitan--Marchenko equation,
goes back over forty years. For example it is mentioned in   
the review by Miura~\cite{Miura}. Dyson~\cite{Dyson} in particular showed
the solution to the Korteweg de Vries equation can be generated from the 
solution to the Gel'fand--Levitan--Marchenko equation along the diagonal.
See for example Drazin and Johnson~\cite[p.~86]{DJ}.
Further results of this nature for other integrable systems 
are summarized in Ablowitz, Ramani and Segur~\cite{ARSII}. 
Then through a sequence of papers P\"oppe~\cite{P-SG,P-KdV,P-KP}, 
P\"oppe and Sattinger~\cite{PS-KP} and Bauhardt and P\"oppe~\cite{BP-ZS},
carried through the programme intimated above. Also in a series
of papers Tracy and Widom, see for example~\cite{TW}, have also
generated similar results. Besides those
already mentioned, the papers by Sato~\cite{SatoI,SatoII},
Segal and Wilson~\cite{SW}, Wilson~\cite{W},
Bornemann~\cite{Btalk}, McKean~\cite{McKean},
Grellier and Gerard~\cite{Gerard} and Beals and Coifman~\cite{BC},
as well as the manuscript by Guest~\cite{MG} were also highly 
influential in this regard.

We note that our second prescription above is analogous 
to that of classical integrable systems and the 
Darboux-dressing transformation. 
The notion of classical integrability in $1+1$ dimensions is synonymous 
with the existence of a Lax pair $(\tilde L,\tilde D)$. 
The Lax pair may consist of differential operators depending 
on the field, i.e.\/ the solution of the associated 
nonlinear integrable partial differential equation, or field valued matrices, 
which can also depend on a spectral parameter. 
The Lax pair satisfies the so called auxiliary linear problem
\begin{equation*}
\tilde L \Psi = \lambda \Psi
\qquad\text{and}\qquad
\partial_t \Psi = \tilde D \Psi. 
\end{equation*}
Here $\Psi$ is called the auxiliary function and $\lambda$ 
is the spectral parameter which is constant in time. 
Compatibility between the two equations above leads to 
the zero curvature condition
\begin{equation*}
\partial_t \tilde L = [\tilde D,\tilde L],
\end{equation*}
which generates the nonlinear integrable equation.
The Darboux-dressing transformation is an efficient and 
elegant way to obtain solutions of the integrable equation
using linear data; see Matveev \& Salle~\cite{MS} 
and Zakharov \& Shabat~\cite{ZS}. 
Let us focus on the $t$-part of the auxiliary linear problem 
to make the connection with our present formulation more concrete.
In the context of integrable systems the Darboux-dressing 
prescription takes the form of a:
(i) Base equation or linearized formulation: $\pa_tP=DP$;
(ii) Auxiliary or modified or dressed equation: $\pa_tQ=\tilde DQ$; and
(iii) Riccati relation or dressing transformation: $P=G\,Q$. 
In the integrable systems frame $D$ is a linear differential operator 
and $\tilde D$ is a nonlinear differential operator that can be 
determined via the dressing process; see  Zakharov \& Shabat~\cite{ZS}
and Drazin and Johnson~\cite{DJ}. The classic example
is the Korteweg de Vries equation, in which case
$D= -4 \partial_x^3$ and 
$\tilde D =-4 \partial_x^3 + 6 u(x,t) \partial_x+ \partial_x u(x,t)$.
In the integrability context extra symmetries 
and thus integrability is provided by the existence of 
the operator $\tilde L$ of the Lax pair. For the 
Korteweg de Vries equation $\tilde L=-\partial^2_x+u(x,t)$.
That the field $u$ satisfies the Korteweg de Vries equation
is ensured by the zero curvature condition.
In our formulation on the other hand, we do not assume 
the existence of a Lax pair as we do not necessarily require 
integrability, thus less symmetry is presupposed. We focus 
on the time part of the Darboux transform described by the 
equations (i)--(iii) just above. They 
yield the equation for the transformation $G$ 
(see also Adamopoulou, Doikou \& Papamikos~\cite{ADP}):
\begin{equation*}
\partial_t G = D\,G - G\,\tilde D.
\end{equation*}
In the present general description the operators $D$ and $\tilde D$ are
known and both linear; at least in all the examples we consider herein.
The operator $G$ turns out to satisfy the associated nonlinear and 
nonlocal partial differential equation just above. 
Depending on the exact form of $\tilde D$ various cases of nonlinearity 
can be considered as will be discussed in detail in what follows. 
Indeed, below we investigate various situations regarding the form of the nonlinear 
operator $\tilde D$, which give rise to qualitatively different 
nonlocal, nonlinear equations. These can be seen as nonlocal 
generalizations of well known examples of integrable equations, 
such as the Korteweg de Vries and nonlinear Schr\"odinger equations and so forth.

Lastly, we remark that Riccati systems play a central role in optimal control theory. 
In particular, the solution to a matrix Riccati equation provides the optimal 
continuous feedback operator in linear-quadratic control. In such systems the 
state is governed by a linear system of equations analogous to those
for $Q$ and $P$ above, and the goal is to optimize a given quadratic cost function.
See for example Martin and Hermann~\cite{MH}, Brockett and Byrnes~\cite{BB} 
and Hermann and Martin~\cite{HM} for more details.

Our paper is structured as follows. In \S\ref{sec:quadflows}
we outline our procedure for generating solutions to
partial differential systems with quadratic nonlocal nonlinearities 
from the corresponding linearized flow. We then examine the 
slightly modified procedure for generating such solutions for 
partial differential systems with cubic and higher odd degree 
nonlocal nonlinearities in \S\ref{sec:generalflows}. 
In \S\ref{sec:examples} we apply our method to a series of
six examples, including a nonlocal reaction-diffusion system, the nonlocal 
Korteweg de Vries equation and 
two nonlocal variants of the nonlinear Schr\"odinger equation,
one with cubic nonlinearity and one with a sinusoidal nonlinearity.
For each of the examples just mentioned we provide numerical simulations 
and details of our numerical methods. Using our method we also
derive an explicit form for solutions to a special case of the 
nonlocal Fisher--Kolmogorov--Petrovskii--Piskunov equation
from biological systems. Finally in \S\ref{sec:conclu} 
we discuss extensions to our method we intend to pursue.
We provide the Matlab programs we used for our simulations
in the supplementary electronic material.

\section{Nonlocal quadratic nonlinearities}\label{sec:quadflows} 
In this section we review and at the same time extend to systems 
our Riccati method for generating solutions to partial differential
equations with quadratic nonlocal nonlinearities. For further 
background details, see Beck \textit{et al.\/}~\cite{BDMS}.
Our basic context is as follows. We suppose we have a separable 
Hilbert space $\Hb$ that admits a direct sum decomposition 
$\Hb=\mathbb{Q}\oplus\mathbb{P}$ into closed subspaces $\Qb$ and $\Pb$.
The set of all subspaces `comparable' in size to $\Qb$ is called
the Fredholm Grassmann manifold $\Gr(\Hb,\mathbb{Q})$. 
Coordinate patches of $\Gr(\Hb,\mathbb{Q})$ 
are graphs of operators $\mathbb{Q}\to\mathbb{P}$ parametrized by, say, $G$. 
See Sato~\cite{SatoII} and Pressley and Segal~\cite{PS} for more details. 

We consider a linear evolutionary flow on the subspace $\Qb$ 
which can be parametrized by two linear operators $Q(t)\colon\Qb\to\Qb$ 
and $P(t)\colon\Qb\to\Pb$ for $t\in[0,T]$ for some $T>0$. More precisely,
we suppose the operator $Q=Q(t)$ is a compact perturbation of the identity, 
and thus a Fredholm operator. Indeed we assume $Q=Q(t)$ has the form $Q=\id+Q^\prime$
where `$\id$' is the identity operator on $\Qb$. We assume for some $T>0$ that  
$Q^\prime\in C^{\infty}\bigl([0,T];\mathfrak J_2(\Qb;\Qb)\bigr)$ and 
$P\in C^{\infty}\bigl([0,T];\mathfrak J_2(\Qb;\Pb)\bigr)$ where 
$\mathfrak J_2(\Qb;\Qb)$ and $\mathfrak J_2(\Qb;\Pb)$ denote the 
class of Hilbert--Schmidt operators from $\Qb\to\Qb$ 
and $\Qb\to\Pb$, respectively. Note that $\mathfrak J_2(\Qb;\Qb)$ 
and $\mathfrak J_2(\Qb;\Pb)$ are Hilbert spaces. Our analysis,
as we see presently, involves two, in general unbounded, 
linear operators $D$ and $B$. In our equations these operators
act on $P$, and since for each $t\in[0,T]$ we would like 
$DP\in\mathfrak J_2(\mathbb Q;\mathbb P)$ 
and $BP\in\mathfrak J_2(\mathbb Q;\mathbb Q)$,
we will assume that 
$P\in C^{\infty}\bigl([0,T];\mathrm{Dom}(D)\cap\mathrm{Dom}(B)\bigr)$.
Here $\mathrm{Dom}(D)\subseteq\mathfrak J_2(\Qb;\Pb)$ and 
$\mathrm{Dom}(B)\subseteq\mathfrak J_2(\Qb;\Pb)$ represent the domains
of $D$ and $B$ in $\mathfrak J_2(\Qb;\Pb)$. Hence in summary, we assume 
\begin{equation*}
P\in C^{\infty}\bigl([0,T];\mathrm{Dom}(D)\cap\mathrm{Dom}(B)\bigr)
\qquad\text{and}\qquad
Q^\prime\in C^{\infty}\bigl([0,T];\mathfrak J_2(\Qb;\Qb)\bigr).
\end{equation*}
Our analysis also involves two bounded linear operators $A=A(t)$
and $C=C(t)$. Indeed we assume that $A\in C^\infty\bigl([0,T];\mathfrak J_2(\Qb;\Qb)\bigr)$
and $C\in C^\infty\bigl([0,T];\mathfrak J_2(\Qb;\Pb)\bigr)$.
We are now in a position to prescribe the evolutionary flow
of the linear operators $Q=Q(t)$ and $P=P(t)$ as follows.
\begin{definition}[Linear Base and Auxiliary Equations]
We assume there exists a $T>0$ such that,
for the linear operators $A$, $B$, $C$ and $D$ 
described above, the linear operators 
$P\in C^{\infty}\bigl([0,T];\mathrm{Dom}(D)\cap\mathrm{Dom}(B)\bigr)$ and 
$Q^\prime\in C^{\infty}\bigl([0,T];\mathfrak J_2(\Qb;\Qb)\bigr)$
satisfy the linear system of operator equations 
\begin{equation*}
\pa_tQ=AQ+BP,  
\qquad\text{and}\qquad
\pa_tP=CQ+DP,
\end{equation*}
where $Q=\id+Q^\prime$. We take $Q^\prime(0)=O$ at time $t=0$ so that $Q(0)=\id$.
We call the evolution equation for $P=P(t)$ the \emph{base equation} 
and the evolution equation for $Q=Q(t)$ the \emph{auxiliary equation}.
\end{definition}
\begin{remark} We note the following:
(i) \emph{Nomenclature:} The base and auxiliary equations 
above are a coupled pair of linear evolution equations for the operators
$P=P(t)$ and $Q=Q(t)$. In many applications and indeed for
all those in this paper $C=O$. In this case the 
equation for $P=P(t)$ collapses to the stand alone equation
$\pa_tP=DP$. For this reason we call it the base equation and
we think of the equation prescribing the evolution of $Q=Q(t)$
as the auxiliary equation; and  
(ii) \emph{In practice:} In all our examples in \S\ref{sec:examples}
we can solve the base and auxiliary equations for $P=P(t)$ 
and $Q=Q(t)$ giving explicit closed form solution expressions 
for all $t\geqslant0$. 
\end{remark}
In addition to the linear base and auxiliary equations
above, we posit a linear relation between $P=P(t)$ and $Q=Q(t)$ as follows.
\begin{definition}[Riccati relation]
We assume there exists a $T>0$ such that, for 
$P\in C^{\infty}\bigl([0,T];\mathrm{Dom}(D)\cap\mathrm{Dom}(B)\bigr)$ and 
$Q^\prime\in C^{\infty}\bigl([0,T];\mathfrak J_2(\Qb;\Qb)\bigr)$,
there exists a linear operator 
$G\in C^{\infty}\bigl([0,T];\mathrm{Dom}(D)\cap\mathrm{Dom}(B)\bigr)$
satisfying the linear Fredholm equation
\begin{equation*}
P=G\,Q,
\end{equation*}
where $Q=\id+Q^\prime$. We call this the \emph{Riccati relation}.
\end{definition}
The existence of a solution to the Riccati relation is governed by the 
regularized Fredholm determinant $\mathrm{det}_2(\id+Q^\prime)$ for the 
Hilbert--Schmidt class operator $Q^\prime=Q^\prime(t)$. For any linear
operator $Q^\prime\in\mathfrak J_2(\Qb;\Qb)$ this regularized Fredholm 
determinant is given by (see Simon~\cite{Simon:Traces} and Reed and Simon~\cite{RSIV})
\begin{equation*}
\mathrm{det}_2\bigl(\id+Q^\prime\bigr)
\coloneqq\exp\Biggl(\sum_{\ell\geqslant2}
\frac{(-1)^{\ell-1}}{\ell}\mathrm{tr}\,(Q^\prime)^\ell\Biggr),
\end{equation*}
where `$\mathrm{tr}$' represents the trace operator. We note that
$\|Q^\prime\|_{\mathfrak J_2(\Qb;\Qb)}^2\equiv\mathrm{tr}\,|Q^\prime|^2$.
The operator $\id+Q^\prime$ is invertible if and only if 
$\mathrm{det}_2\bigl(\id+Q^\prime\bigr)\neq0$; again see Simon~\cite{Simon:Traces} 
and Reed and Simon~\cite{RSIV} for more details.
\begin{lemma}[Existence and Uniqueness: Riccati relation]\label{lemma:EandU}
Assume there exists a $T>0$ such that
$P\in C^{\infty}\bigl([0,T];\mathrm{Dom}(D)\cap\mathrm{Dom}(B)\bigr)$,
$Q^\prime\in C^{\infty}\bigl([0,T];\mathfrak J_2(\Qb;\Qb)\bigr)$
and $Q^\prime(0)=O$. Then there exists a $T^\prime>0$ with $T^\prime\leqslant T$ 
such that for $t\in[0,T^\prime]$ we have 
$\mathrm{det}_2\bigl(\id+Q^\prime(t)\bigr)\neq0$ and
$\|Q^\prime(t)\|_{\mathfrak J_2(\Qb;\Qb)}<1$.
In particular, there exists a unique solution 
$G\in C^{\infty}\bigl([0,T^\prime];\mathrm{Dom}(D)\cap\mathrm{Dom}(B)\bigr)$
to the Riccati relation.
\end{lemma}
\begin{proof}
Since $Q^\prime\in C^{\infty}\bigl([0,T];\mathfrak J_2(\Qb;\Qb)\bigr)$
and $Q^\prime(0)=O$, by continuity there exists a $T^\prime>0$ 
with $T^\prime\leqslant T$ such that $\|Q^\prime(t)\|_{\mathfrak J_2(\Qb;\Qb)}<1$ 
for $t\in[0,T^\prime]$. Similarly by continuity, since $Q^\prime(0)=O$,
for a short time at least we expect $\mathrm{det}_2\bigl(\id+Q^\prime\bigr)\neq0$.
We can however assess this as follows. Using the 
regularized Fredholm determinant formula above, we observe that
\begin{equation*}
\Bigl|\mathrm{det}_2\bigl(\id+Q^\prime\bigr)-1\Bigr|
\leqslant \sum_{n\geqslant 1}\frac{1}{n!}
\biggl(\sum_{\ell\geqslant2}\frac{1}{\ell}\mathrm{tr}\,\bigl|Q^\prime\bigr|^\ell\biggr)
\leqslant\exp\biggl(\sum_{\ell\geqslant2}\frac{1}{\ell}\|Q^\prime\|_{\mathfrak J_2(\Qb;\Qb)}^\ell\biggr)-1.
\end{equation*}
In the last step we used that 
$\mathrm{tr}\,\bigl|Q^\prime\bigr|^\ell\leqslant 
\bigl(\mathrm{tr}\,\bigl|Q^\prime\bigr|^2\bigr)^{\ell/2}$ 
for all $\ell\geqslant2$.
The series in the exponent in the final term above converges if 
$\|Q^\prime\|_{\mathfrak J_2(\Qb;\Qb)}<1$. We deduce that provided 
$\|Q^\prime\|_{\mathfrak J_2(\Qb;\Qb)}$ is sufficiently small then 
its regularized Fredholm determinant is bounded away from zero.
By continuity there exists a $T^\prime$, possibly smaller than the 
choice above, such that for all $t\in[0,T^\prime]$ we know  
$\|Q^\prime(t)\|_{\mathfrak J_2(\Qb;\Qb)}$ 
is sufficiently small and the determinant is bounded away from zero.

Next, we set $\|H\|_{\mathrm{Dom}(D)\cap\mathrm{Dom}(B)}
\coloneqq\|DH\|_{\mathfrak J_2(\Qb;\Pb)}+\|BH\|_{\mathfrak J_2(\Qb;\Pb)}$
for any $H\in\mathrm{Dom}(D)\cap\mathrm{Dom}(B)$,
while $\|\,\cdot\,\|_{\mathrm{op}}$ denotes the operator norm for
bounded operators on $\Qb$. We observe that for any $n\in\mathbb N$ we have  
\begin{align*}
\Bigl\|P(t)\bigl(Q^\prime(t)\bigr)^n\Bigr\|_{\mathrm{Dom}(D)\cap\mathrm{Dom}(B)}
&\leqslant\|P(t)\|_{\mathrm{Dom}(D)\cap\mathrm{Dom}(B)}
\Bigl\|\bigl(Q^\prime(t)\bigr)^n\Bigr\|_{\mathrm{op}}\\
&\leqslant\|P(t)\|_{\mathrm{Dom}(D)\cap\mathrm{Dom}(B)}
\|Q^\prime(t)\|^n_{\mathrm{op}}\\
&\leqslant\|P(t)\|_{\mathrm{Dom}(D)\cap\mathrm{Dom}(B)}
\|Q^\prime(t)\|^n_{\mathfrak J_2(\Qb;\Qb)}.
\end{align*}
Hence we observe that 
\begin{align*}
\biggl\|P(t)\biggl(\id+\sum_{n\geqslant1}(-1)^n
&\bigl(Q^\prime(t)\bigr)^n\biggr)\biggr\|_{\mathrm{Dom}(D)\cap\mathrm{Dom}(B)}\\
&\leqslant
\|P(t)\|_{\mathrm{Dom}(D)\cap\mathrm{Dom}(B)}\biggl(1
+\sum_{n\geqslant1}\|Q^\prime(t)\|^n_{\mathfrak J_2(\Qb;\Qb)}\biggr)\\
&\leqslant
\|P(t)\|_{\mathrm{Dom}(D)\cap\mathrm{Dom}(B)}\bigl(1-\|Q^\prime(t)\|_{\mathfrak J_2(\Qb;\Qb)}\bigr)^{-1}.
\end{align*}
Hence using the operator series expansion for $(\id+Q^\prime(t))^{-1}$ 
we observe we have established that 
\begin{equation*}
\Bigl\|P(t)\bigl(\id+Q^\prime(t)\bigr)^{-1}\Bigr\|_{\mathrm{Dom}(D)\cap\mathrm{Dom}(B)}
\leqslant
\|P(t)\|_{\mathrm{Dom}(D)\cap\mathrm{Dom}(B)}\bigl(1-\|Q^\prime(t)\|_{\mathfrak J_2(\Qb;\Qb)}\bigr)^{-1}.
\end{equation*}
Hence there exists a $T^\prime>0$ such that 
for each $t\in[0,T^\prime]$ we know $G(t)=P(t)\bigl(\id+Q^\prime(t)\bigr)^{-1}$
exists, is unique, and in fact 
$G\in C^{\infty}\bigl([0,T^\prime];\mathrm{Dom}(D)\cap\mathrm{Dom}(B)\bigr)$. \qed
\end{proof}
\begin{remark}[Initial data]
We have already remarked that we set $Q^\prime(0)=O$ so that $Q(0)=\id$.
Consistent with the Riccati relation we hereafter set $P(0)=G(0)$. 
\end{remark}
Our first main result in this section is as follows.
\begin{theorem}[Quadratic Degree Evolution Equation]\label{theorem:quadmain}
Given initial data $G_0\in\mathrm{Dom}(D)\cap\mathrm{Dom}(B)$
we set $Q^\prime(0)=O$ and $P(0)=G_0$.
Suppose there exists a $T>0$ such that the linear operators 
$P\in C^{\infty}\bigl([0,T];\mathrm{Dom}(D)\cap\mathrm{Dom}(B)\bigr)$ and
$Q^\prime\in C^{\infty}\bigl([0,T];\mathfrak J_2(\Qb;\Qb)\bigr)$ satisfy the
linear base and auxiliary equations. We choose $T>0$ so that for 
$t\in[0,T]$ we have $\mathrm{det}_2\bigl(\id+Q^\prime(t)\bigr)\neq0$ and
$\|Q^\prime(t)\|_{\mathfrak J_2(\Qb;\Qb)}<1$.
Then there exists a unique solution 
$G\in C^{\infty}\bigl([0,T];\mathrm{Dom}(D)\cap\mathrm{Dom}(B)\bigr)$
to the Riccati relation which necessarily satisfies $G(0)=G_0$ 
and the Riccati evolution equation
\begin{equation*}
\pa_tG=C+DG-G\,(A+BG). 
\end{equation*}
\end{theorem}
\begin{proof}
By direct computation, differentiating the Riccati relation $P=G\,Q$ 
with respect to time using the product rule, using the base and 
auxiliary equations and feeding back through the Riccati relation, 
we find
$\bigl(\pa_tG\bigr)Q=\pa_tP-G\,\pa_tQ=(C+DG)\,Q-\bigl(G\,(A+BG)\bigr)\,Q$. 
Equivalencing with respect to $Q$, i.e.\/ postcomposing by $Q^{-1}$,
establishes the result.\qed
\end{proof}
\begin{remark}
We assume throughout this paper that $C=C(t)$ is a bounded operator,
indeed that $C\in C^\infty\bigl([0,T];\mathfrak J_2(\Qb;\Pb)\bigr)$. In fact
in every application in \S\ref{sec:examples} we take $C=O$. However
in general $C=C(t)$ would represent some non-homogeneous forcing
in the Riccati equation satisfied by $G=G(t)$. Further, in 
Doikou, Malham \& Wiese~\cite{DMW} we apply our methods here
to stochastic partial differential equations. One example therein
features additive space-time white noise. In that
case the term $C=C(t)$ represents the non-homogenous space-time 
white noise forcing term and we must thus allow for $C=C(t)$ 
to be an unbounded operator. 
\end{remark}
We now turn our attention to applications of Theorem~\ref{theorem:quadmain}
above and demonstrate how to find solutions to a large class of partial
differential systems with nonlocal quadratic nonlinearities. Guided by
our results above, we now suppose the classes of operators we have considered 
thusfar to be those with integral kernels on $\R\times\R$. For $x,y\in\R$ and 
$t\geqslant0$, suppose the functions $p=p(x,y;t)$ and $q^\prime=q^\prime(x,y;t)$ 
are matrix valued, with $p\in\R^{n^\prime\times n}$ and $q^\prime\in\R^{n\times n}$
for some $n,n^\prime\in\mathbb N$, and they satisfy the linear base 
and auxiliary equations
\begin{equation*}
\pa_t p(x,y;t)=d(\partial_1)\,p(x,y;t)
\qquad\text{and}\qquad
\pa_t q^\prime(x,y;t)=b(x)\,p(x,y;t).
\end{equation*}
Here the unbounded operator $d=d(\pa_1)$ is a constant coefficient 
scalar polynomial function of the partial differential operator with respect to
the first component $\pa_1$, while $b=b(x)$ is a smooth 
bounded square-integrable $\R^{n\times n^\prime}$-valued function of $x\in\R$. 
We can explicitly solve these equations for $p=p(x,y;t)$ and 
$q^\prime=q^\prime(x,y;t)$ in terms of their Fourier transforms as follows.
Note we use the following notation for the Fourier transform of any 
function $f=f(x,y)$ and its inverse: 
\begin{align*}
\widehat{f}(k,\kappa)&\coloneqq\int_{\R^2} f(x,y)\mathrm{e}^{2\pi\mathrm{i}(kx+\kappa y)}\,\rd x\,\rd y
\intertext{and}
f(x,y)&\coloneqq\int_{\R^2}\widehat{f}(k,\kappa)\mathrm{e}^{-2\pi\mathrm{i}(kx+\kappa y)}\,\rd k\,\rd\kappa.
\end{align*}
\begin{lemma}\label{lemma:explicitlinearquadratic}
Let $\widehat{p}=\widehat{p}(k,\kappa;t)$ and 
$\widehat{q}^\prime=\widehat{q}^\prime(k,\kappa;t)$ denote the
two-dimensional Fourier transforms of the solutions to the 
linear base and auxiliary equations just above.
Assume that $q^\prime(x,y;0)\equiv0$ and $p(x,y;0)=p_0(x,y)$.
Then for all $t\geqslant0$ the functions $\widehat{p}$ and 
$\widehat{q}^\prime$ are explicitly given by
\begin{align*}
\widehat{p}(k,\kappa;t)&=\exp\bigl(d(2\pi\mathrm{i}k)\,t\bigr)\,\widehat{p}_0(k,\kappa)
\intertext{and}
\widehat{q}^\prime(k,\kappa;t)
&=\int_\R\widehat{b}(k-\lambda)\,\widehat{I}(\lambda;t)
\,\widehat{p}_0(\lambda,\kappa)\,\rd\lambda,
\end{align*}
where $\widehat{I}(k;t)
\coloneqq\bigl(\exp\bigl(d(2\pi\mathrm{i}k)\,t\bigr)-1\bigr)/d(2\pi\mathrm{i}k)$
and indeed $q^\prime(x,y;t)=b(x)\int_\R I(x-z,t)\, p_0(z,y)\,\rd z$.
\end{lemma}
\begin{proof}
Taking the two-dimensional Fourier transform of the base equation 
we generate the decoupled equation 
$\pa_t\widehat{p}(k,\kappa;t)=d(2\pi\mathrm{i}k)\widehat{p}(k,\kappa;t)$
whose solution is the form for $\widehat{p}(k,\kappa;t)$ shown.
Then take the Fourier transform of the 
auxiliary equation to generate the equation 
$\pa_t\widehat{q}^\prime(k,\kappa;t)
=\int_\R\widehat{b}(k-\lambda)\,\widehat{p}(\lambda,\kappa;t)\,\rd\lambda$. 
Substituting in the explicit form for $\widehat{p}=\widehat{p}(k,\kappa;t)$
and integrating with respect to time, using 
$\widehat{q}^\prime(k,\kappa;0)=0$, generates the form for 
$\widehat{q}^\prime=\widehat{q}^\prime(k,\kappa;t)$ shown.\qed
\end{proof}
\begin{remark}[Hilbert--Schmidt solutions]\label{remark:HS}
We suppose here the separable Hilbert space 
$\mathbb H=L^2(\mathbb R;\mathbb R^n)\times(\mathrm{Dom}(D)\cap\mathrm{Dom}(B))$ 
with
$\mathrm{Dom}(D)\cap\mathrm{Dom}(B)\subseteq L^2(\mathbb R;\mathbb R^{n^\prime})$
where $n$ and $n^\prime$ are the 
dimensions above. Then $\mathbb P$ and $\mathbb Q$ are closed 
subspaces in the direct sum decomposition $\mathbb H=\mathbb Q\oplus\mathbb P$;
see Beck \textit{et al.\/} ~\cite{BDMS}.
The functions in $\Qb$ are $\R^n$-valued while those in $\Pb$ are $\R^{n^\prime}$-valued.
By standard theory, $Q^\prime(t)\in\mathfrak J_2(\Qb;\Qb)$ 
and $P(t)\in\mathfrak J_2(\Qb;\Pb)$ if and only if there exist kernel 
functions $q^\prime(\cdot,\cdot;t)\in L^2(\R^2;\R^{n\times n})$ and 
$p(\cdot,\cdot;t)\in L^2(\R^2;\R^{n^\prime\times n})$ with
the action of $Q^\prime(t)$ and $P(t)$ given through $q^\prime$ and $p$, 
respectively. Further we know that 
$\|Q^\prime(t)\|_{\mathfrak J_2(\Qb;\Qb)}=\|q^\prime(\cdot,\cdot;t)\|_{L^2(\R^2;\R^{n\times n})}$ 
and $\|P(t)\|_{\mathfrak J_2(\Qb;\Pb)}=\|p(\cdot,\cdot;t)\|_{L^2(\R^2;\R^{n^\prime\times n})}$.
For more details see for example Reed \& Simon~\cite[p.~210]{RS} or Karambal \& Malham~\cite{KM}.
The linear base and auxiliary equations above correspond to the 
case when $A=C=O$, $D=d(\pa_1)$ and $B$ is given by the bounded multiplicative
operator $b=b(x)$. Recall that in our ``abstract'' formulation above we required
that $P\in C^\infty\bigl([0,T];\mathrm{Dom}(D)\cap\mathrm{Dom}(B)\bigr)$. 
The explicit form for $p=p(x,y;t)$ given in 
Lemma~\ref{lemma:explicitlinearquadratic} reveals that 
$P$ will only have this property for certain classes of operators $d=d(\pa_1)$.
For example suppose $d=d(\pa_1)$ is diffusive so that it takes the form of
a polynomial with only even degree terms in $\pa_1$ and the real scalar coefficient
of the degree $2N$ term is of the form $(-1)^{N+1}\alpha_{2N}$. In this case 
the exponential term $\exp\bigl(d(2\pi\mathrm{i}k)\,t\bigr)$ decays exponentially 
for all $t>0$. We could also include dispersive forms for $d$. For example 
$d=\pa_1^3$, for which the exponential term $\exp\bigl(d(2\pi\mathrm{i}k)\,t\bigr)$ 
remains bounded for all $t>0$. We also note that for such diffusive or dispersive
forms for $d=d(\pa_1)$ the integral kernel function $p=p(x,y;t)$ is in fact
smooth. 
Also recall from our ``abstract'' formulation 
we require $Q^\prime\in C^\infty\bigl([0,t];\mathfrak J_2(\Qb;\Qb)\bigr)$.
The explicit form for $q^\prime=q^\prime(x,y;t)$ given in 
Lemma~\ref{lemma:explicitlinearquadratic} reveals that 
its time dependence is characterized through the term $\widehat{I}(k;t)$.
For the diffusive or dispersive forms for $d=d(\pa_1)$ just discussed
we observe that $\widehat{I}(k;t)\to-1/d(2\pi\mathrm{i}k)$ for all $k\neq0$
while for the singular value $k=0$ the term $\widehat{I}(0;t)$ grows 
linearly in time. Thus in such cases, while we know that for some
time $T>0$ for $t\in[0,T]$ we have
$\|Q^\prime(t)\|_{\mathfrak J_2(\Qb;\Qb)}=\|q^\prime(\cdot,\cdot;t)\|_{L^2(\R^2;\R^{n\times n})}
=\|\widehat{q}^\prime(\cdot,\cdot;t)\|_{L^2(\R^2;\C^{n\times n})}$ is bounded,
we also have 
\begin{align*}
\|&\widehat{q}^\prime(\cdot,\cdot;t)\|_{L^2(\R^2;\C^{n\times n})}\\
&=\int_{\R^4}\widehat{p}_0^\ast(\lambda,\kappa)\,\widehat{I}^\ast(\lambda;t)\,
\widehat{b}^\ast(k-\lambda)\,\widehat{b}(k-\nu)\,\widehat{I}(\nu;t)
\,\widehat{p}_0(\nu,\kappa)\,\rd\lambda\,\rd\nu\,\rd\kappa\,\rd k\\
&\leqslant
\biggl\|\int_\R\widehat{b}_0^\ast(k-\cdot)\,
\widehat{b}_0(k-\cdot)\,\rd k\biggr\|_{L^\infty(\R^2;\R^{n\times n})}\cdot\\
&\qquad\qquad\qquad\qquad\cdot\biggl\|\int_\R\widehat{p}_0^\ast(\cdot,\kappa)\,
\widehat{p}_0(\cdot,\kappa)\,\rd\kappa\biggr\|_{L^\infty(\R^2;\R^{n\times n})}
\bigl\|\widehat{I}(t)\bigr\|_{L^1(\R;\C)}^2.
\end{align*}
Hence provided the terms on the right are bounded with 
$\bigl\|\widehat{I}(t)\bigr\|_{L^1(\R;\C)}$
bounded for all $t>0$, then 
$\|\widehat{q}^\prime(\cdot,\cdot;t)\|_{L^2(\R^2;\C^{n\times n})}$
will be bounded for all $t>0$, and indeed smooth. 
However how far the interval of time on which 
$\mathrm{det}_2\bigl(\id+Q^\prime(t)\bigr)\neq0$ and
$\|Q^\prime(t)\|_{\mathfrak J_2(\Qb;\Qb)}<1$ extends, for now,
we treat on case by case basis.
\end{remark}
\begin{corollary}[Evolutionary PDEs with quadratic nonlocal nonlinearities]
Given initial data 
$g_0\in C^\infty(\R^2;\R^{n^\prime\times n})\cap L^2(\R^2;\R^{n^\prime\times n})$
for some $n,n^\prime\in\mathbb N$, suppose $p=p(x,y;t)$ and $q^\prime=q^\prime(x,y;t)$
are the solutions to the linear base and auxiliary equations
from Lemma~\ref{lemma:explicitlinearquadratic} for which $p_0\equiv g_0$ 
and $q^\prime(x,y;0)\equiv0$. Let $\mathrm{Dom}(d)$ denote the domain
of the operator $d=d(\pa_1)$ and suppose it is of the diffusive or
dispersive form described in Remark~\ref{remark:HS}.
Then there exists a $T>0$ such that the solution 
$g\in C^\infty\bigl([0,T];\mathrm{Dom}(d)\cap L^2(\R^2;\R^{n^\prime\times n})\bigr)$ 
to the linear Fredholm equation
\begin{equation*}
p(x,y;t)=g(x,y;t)+\int_{\R}g(x,z;t)\,q^\prime(z,y;t)\,\rd z
\end{equation*}
solves the evolutionary partial differential equation with 
quadratic nonlocal nonlinearities of the form
\begin{equation*}
\pa_tg(x,y;t)=d(\pa_x)\,g(x,y;t)-\int_{\R}g(x,z;t)\,b(z)\,g(z,y;t)\,\rd z.
\end{equation*}
\end{corollary}
\begin{proof}
That for some $T>0$ there exists a solution 
$g\in C^\infty\bigl([0,T];\mathrm{Dom}(d)\cap L^2(\R^2;\R^{n^\prime\times n})\bigr)$ 
to the linear Fredholm equation (Riccati relation) shown is a consequence 
of Lemma~\ref{lemma:EandU} and Remark~\ref{remark:HS}. The solution $g$
is the integral kernel of $G$. That this solution $g$ to the Riccati relation
solves the evolutionary partial differential equation with the quadratic 
nonlocal nonlinearity shown is a direct consequence of the Quadratic Degree
Evolution Equation Theorem~\ref{theorem:quadmain}. \qed
\end{proof}
\begin{remark}\label{remark:directview}
We can also now think of this result in the following way.
First differentiate the above linear Fredholm equation in 
the Corollary with respect to time using the product rule, and  
use that $p$ and $q^\prime$ satisfy the linear base 
and auxiliary equations so that
\begin{multline*}
\pa_tg(x,y;t)+\int_{\R}\pa_tg(x,z;t)\,q^\prime(z,y;t)\,\rd z\\
=d(\pa_1)p(x,y;t)-\int_{\R}g(x,z;t)\,b(z)p(z,y;t)\,\rd z.
\end{multline*}
Second replacing all instances of $p$ using the linear Fredholm
equation above and swapping integration labels we obtain
\begin{align*}
\pa_tg(x,y;t)+&\int_{\R}\pa_tg(x,z;t)\,q^\prime(z,y;t)\,\rd z\\
=&\;d(\pa_x)g(x,y;t)+\int_{\R}d(\pa_x)g(x,z;t)\,q^\prime(z,y;t)\,\rd z\\
&\;-\int_{\R}g(x,z;t)\,b(z)g(z,y;t)\,\rd z\\
&\;-\int_{\R}\biggl(\int_{\R}g(x,\zeta;t)\,b(\zeta)g(\zeta,z)\,\rd\zeta\biggr)
q^\prime(z,y;t)\,\rd z.
\end{align*}
We can express this in the form
\begin{multline*}
\int_{\R}\biggl(\pa_tg(x,z;t)-d(\pa_x)g(x,z;t)\\
+\int_{\R}g(x,\zeta;t)\,b(\zeta)g(\zeta,z;t)\,\rd\zeta\biggr)
\bigl(\delta(z-y)+q^\prime(z,y;t)\bigr)\,\rd z=0.
\end{multline*}
Third we postmultiply by `$\delta(y-\eta)+\tilde{q}^\prime(y,\eta;t)$'
for some $\eta\in\R$. This is the kernel corresponding to the inverse operator 
$\id+\tilde Q^\prime$ of $\id+Q^\prime$. Integrating over $y\in\R$ gives the result
for $g=g(x,\eta;t)$. This derivation follows that in  
Beck \textit{et al.\/}~\cite{BDMS} 
for scalar partial differential equations. 
\end{remark}
\begin{remark}
Some observations are as follows:
(i) \emph{Nonlocal nonlinearities with derivatives:}
Starting with the linear base and auxiliary equations
for $p=p(x,y;t)$ and $q^\prime=q^\prime(x,y;t)$, we could have 
taken $b$ to be any constant coefficient polynomial of $\pa_1$. 
With minor modifications, all of the main arguments above still apply.
Our explicit solution for $q^\prime=q^\prime(x,y;t)$ will be slightly
more involved. One of our examples in \S\ref{sec:examples} is
the nonlocal Korteweg de Vries equation for which $b=\pa_1$; 
(ii) \emph{Smooth solutions:} All derivatives are with
respect to the first parameter $x$. Differentiating the Riccati relation gives
$\pa_xp(x,y;t)=\pa_xg(x,y;t)+\int_{\R}\pa_xg(x,z;t)\,q^\prime(z,y;t)\,\rd z$.
Hence the regularity of the solution $g$ is directly determined by
the regularity of the solution of the base equation $p$ for the time
the Riccati relation is solvable, in particular while  
$\mathrm{det}_2\bigl(\id+Q^\prime(t)\bigr)\neq0$ and
$\|Q^\prime(t)\|_{\mathfrak J_2(\Qb;\Qb)}<1$. Hence if 
$p$ is smooth on this interval, then the solution $g$ is smooth
on this interval;
(iii) \emph{Time as a parameter:} Importantly, when we can explicitly 
solve for $p=p(x,y;t)$ and $q^\prime=q^\prime(x,y;t)$, as we do above, 
then time $t$ plays the role of a parameter. We choose the
time at which we wish to compute the solution and we solve the 
linear Fredholm equation to generate the solution $g$ for that time $t$;   
(iv) \emph{Non-homogeneous coefficients:} 
In principle, if $d$ and $b$ are polynomials of $\pa_x$, 
the coefficients in these polynomial could also be functions of $x$.
Though we can in principle always find series solutions to the 
linear base and auxiliary equations, we would now have the issue 
as to whether we can derive explicit formulae for $p$ and $q^\prime$.
In such cases we may need to evaluate a series or 
numerically integrate in time to obtain $p$ and $q^\prime$. 
Thus we cannot compute solutions as simply as in the sense 
outlined in Item~(iii) just above. An important example is that
of evolutionary stochastic partial differential equations
with non-local nonlinearities. The presence of Wiener fields 
in such equations as non-homogeneous additive terms 
or multiplicative factors means that the base equation
must be solved numerically. For example the base equation 
might be the stochastic heat equation. See Doikou, Malham
\& Wiese~\cite{DMW} for more details;
(v) \emph{Complex valued solutions:} In general 
$g$ could be complex matrix valued; see \S\ref{sec:generalflows} next;
(vi) \emph{Domains}: If $x,y\in\mathbb I$ where $\mathbb I$ is a 
finite or semi-infinite interval on $\R$, then the above calculations
go through, see Beck~\textit{et al.\/}~\cite{BDMS} and 
also Doikou~\textit{et al.\/}~\cite{DMW} where $\mathbb I=\mathbb T$, 
the torus with period $2\pi$; and
(vii) \emph{Multi-dimensional domains:} 
If $x,y\in\R^n$ for some $n\in\mathbb{N}$ and $d=d(\Delta_1)$ 
is a polynomial function of the Laplacian acting on the first argument,
then in principle the calculations above go through; 
see our Conclusions~\S\ref{sec:conclu}.
\end{remark}

\section{Nonlocal cubic and higher odd degree nonlinearities}
\label{sec:generalflows} 
We assume the same set-up as in the first two paragraphs in
\S\ref{sec:quadflows} up to the point when we discuss
the unbounded linear operator $D$. In this section we  
assume $\Pb\subseteq\Qb$.
We still assume that $D$ is in general an unbounded, linear operator, 
however we set $B=O$ and $C=O$ while $A$ is a bounded operator which we 
discuss presently.
We assume there exists a $T>0$ such that for each $t\in[0,T]$ we have 
$P\in C^{\infty}\bigl([0,T];\mathrm{Dom}(D)\bigr)$ and 
$Q^\prime\in C^{\infty}\bigl([0,T];\mathfrak J_2(\Qb;\Qb)\bigr)$.
Our analysis in this section also involves the bounded linear operator 
$A\in\mathfrak J_2(\Qb;\Qb)$ which depends on another bounded
linear operator as follows. For a known operator $H\in\mathfrak J_2(\Qb;\Pb)$ 
we assume $A$ has the form $A=f(HH^\dag)$ where the function $f$ is given by
\begin{equation*}
f(x)=\mathrm{i}\sum_{m\geqslant0}\alpha_mx^m,
\end{equation*}
where $\mathrm{i}=\sqrt{-1}$ and the $\alpha_m$ are real coefficients. 
Note $H^\dag$ denotes the operator adjoint to $H$.
We further assume this power series expansion has an infinite radius 
of convergence. In this section we assume the evolutionary flow
of the linear operators $Q=Q(t)$ and $P=P(t)$ is as follows.
\begin{definition}[Linear Base and Auxiliary Equations (modified)]
\label{definition:odddegreebase}
We assume there exists a $T>0$ such that for the linear operators
$A$ and $D$ described above, the linear operators 
$P\in C^{\infty}\bigl([0,T];\mathrm{Dom}(D)\bigr)$ and 
$Q^\prime\in C^{\infty}\bigl([0,T];\mathfrak J_2(\Qb;\Qb)\bigr)$
satisfy the linear system of operator equations
\begin{equation*}
\pa_tP=DP,
\qquad\text{and}\qquad
\pa_tQ=f(PP^\dag)\,Q,  
\end{equation*}
where $Q=\id+Q^\prime$. We take $Q^\prime(0)=O$ at time $t=0$ 
so that $Q(0)=\id$. We call the evolution equation for $P=P(t)$
the \emph{base equation} and the evolution equation for $Q=Q(t)$ 
the \emph{auxiliary equation}.
\end{definition}
\begin{remark}
Note we first solve the base equation for 
$P\in C^{\infty}\bigl([0,T];\mathrm{Dom}(D)\bigr)$.
Then with $P$ given, we observe that $f=f(PP^\dag)$ is a
given linear operator in the auxiliary equation.
\end{remark}
\begin{lemma}
Assume for some $T>0$ that $P\in C^{\infty}\bigl([0,T];\mathrm{Dom}(D)\bigr)$ and 
$Q^\prime\in C^{\infty}\bigl([0,T];\mathfrak J_2(\Qb;\Qb)\bigr)$
satisfy the linear base and auxiliary equations above. 
Then $Q(0)=\id$ implies $QQ^\dag=\id$ for all $t\in[0,T]$. 
\end{lemma}
\begin{proof}
By definition $f^\dag=-f$, and using the product rule 
$\pa_t\bigl(QQ^\dag\bigr)=f\,(QQ^\dag)-(QQ^\dag)\,f$. 
Thus $QQ^\dag=\id$ is a fixed point of this flow and $Q(0)=\id$ implies 
$QQ^\dag=\id$ for all $t\in[0,T]$.\qed
\end{proof}
In addition to the linear base and auxiliary equations above, we 
again posit a linear relation between $P=P(t)$ and $Q=Q(t)$, the 
Riccati relation $P=G\,Q$, exactly as in \S\ref{sec:quadflows}.
Indeed the results of Lemma~\ref{lemma:EandU} for the existence and 
uniqueness of a solution $G$ to the Riccati relation apply here. 
Further, as previously, hereafter we set $P(0)=G(0)$.
Our main result of this section is as follows.
\begin{theorem}[Odd Degree Evolution Equation]\label{theorem:generalmain}
Given initial data $G_0\in\mathrm{Dom}(D)$
we set $Q(0)=\id$ and $P(0)=G_0$.
Suppose there exists a $T>0$ such that the linear operators 
$P\in C^{\infty}\bigl([0,T];\mathrm{Dom}(D)\bigr)$ and
$Q-\id\in C^{\infty}\bigl([0,T];\mathfrak J_2(\Qb;\Qb)\bigr)$ satisfy the
linear base and auxiliary equations above. We choose $T>0$ so that for 
$t\in[0,T]$ we have $\mathrm{det}_2\bigl(Q(t)\bigr)\neq0$ and
$\|Q^\prime(t)\|_{\mathfrak J_2(\Qb;\Qb)}<1$. Then there exists a unique solution 
$G\in C^{\infty}\bigl([0,T];\mathrm{Dom}(D)\bigr)$
to the Riccati relation which necessarily satisfies the evolution equation
\begin{equation*}
\pa_tG=DG-G\,f(GG^\dag).
\end{equation*}
\end{theorem} 
\begin{proof}
First, using the Riccati relation and that $QQ^\dag=\id$, we have
$PP^\dag=GG^\dag$ and thus $f(PP^\dag)=f(GG^\dag)$ for all $t\in[0,T]$. 
Second, differentiating the Riccati relation with respect to time using the product rule
and then substituting for $P$ using the Riccati relation,
we have $\bigl(\pa_tG\bigr)Q=\pa_tP-G\,\pa_tQ=DG\,Q-G\,f(PP^\dag)\,Q=DG\,Q-G\,f(GG^\dag)\,Q$.
As previously, equivalencing by $Q$, i.e.\/ postcomposing by $Q^{-1}$, 
establishes the result. \qed
\end{proof}
We now consider applications of Theorem~\ref{theorem:generalmain} above 
and demonstrate how to find solutions 
to classes of partial differential systems with nonlocal odd degree nonlinearities. 
For $x,y\in\R$ and $t\geqslant0$, suppose the functions $p=p(x,y;t)$ and $q=q(x,y;t)$ 
are scalar complex valued, with $p\in\C$ and $q\in\C$, and they satisfy the linear base 
and auxiliary equations 
\begin{equation*}
\pa_t p=-\mathrm{i}h(\partial_1)p
\qquad\text{and}\qquad
\pa_t q=f^\star\bigl(p\star p^\dag\bigr)\star q.
\end{equation*}
Here $h=h(\pa_1)$ is a polynomial function of $\pa_1$ 
with only even degree terms of its argument and constant coefficients.
By analogy with \S\ref{sec:quadflows}, here we have made the choice 
$d(\pa_1)=-\mathrm{i}h(\partial_1)$. 
The nonlocal product `$\star$' is defined for any two 
functions $w,w^\prime\in L^2(\R^2;\C)$ by
\begin{equation*}
\bigl(w\star w^\prime\bigr)(x,y)\coloneqq
\int_\R w(x,z)\,w^\prime(z,y)\,\rd z.
\end{equation*}
Hence the expression $p\star p^\dag$ thus represents the kernel function 
\begin{equation*}
\bigl(p\star p^\dag\bigr)(x,y;t)\coloneqq
\int_\R p(x,z;t)p^*(y,z;t)\,\rd z,
\end{equation*}
Note here we have used that if an operator has integral kernel $p=p(x,y;t)$, 
its adjoint has integral kernel $p^\ast(y,x;t)$, where the `$\ast$' in general 
denotes complex conjugate transpose. The expression $f^\star(c)$, 
for some kernel function $c$, represents the series 
with real coefficients $\alpha_m$ given by 
\begin{equation*}
f^\star(c)=\mathrm{i}\sum_{m\geqslant0}\alpha_m c^{\star m},
\end{equation*}
where $c^{\star m}$ is the $m$-fold product $c\star\cdots\star c$. 
We assume this power series has an infinite radius of convergence.
In the linear auxiliary equation we take $c=p\star p^\dag$.
It is natural to take the Fourier transform of the 
base and auxiliary equations with respect to $x$ and $y$. 
The corresponding 
equations for $\widehat{p}=\widehat{p}(k,\kappa;t)$ 
and $\widehat{q}=\widehat{q}(k,\kappa;t)$ are
\begin{equation*}
\pa_t \widehat{p}=-\mathrm{i}h(2\pi\mathrm{i}k)\,\widehat{p}
\qquad\text{and}\qquad
\pa_t \widehat{q}=\widehat{f}^\star\bigl(\widehat{p}\star\widehat{p}^\dag\bigr)
\star\widehat{q}.
\end{equation*}
Here we have used Parseval's identity for Fourier transforms which implies
\begin{equation*}
\bigl(\widehat{w\star w^\prime}\bigr)(k,\kappa)
=\int_\R\widehat{w}(k,\lambda)\,\widehat{w}^\prime(\lambda,\kappa)\,\rd\lambda
=\bigl(\widehat{w}\star\widehat{w}^\prime\bigr)(k,\kappa)
\end{equation*}
for any two functions $w,w^\prime\in L^2(\R^2;\C)$.
Hence we see that for $f^\star=f^\star(c)$, we have
\begin{equation*}
f^\star=\mathrm{i}\sum_{m\geqslant0}\alpha_m c^{\star m}
\qquad\Leftrightarrow\qquad
\widehat{f}^\star=\mathrm{i}\sum_{m\geqslant0}\alpha_m \widehat{c}^{\star m}.
\end{equation*}
Further we note that if $q(x,y;t)=\delta(x-y)+q^\prime(x,y;t)$ then 
$\widehat{q}(k,\kappa;t)=\delta(k-\kappa)+\widehat{q}^\prime(k,\kappa;t)$.
The Dirac delta function $\delta$ here also represents the identity 
with respect to the `$\star$' product so that for any $w\in L^2(\R^2;\C)$ 
we have $w\star\delta=\delta\star w=w$.
With all this in hand, we can in fact explicitly solve for 
$\widehat{p}=\widehat{p}(k,\kappa;t)$ and 
$\widehat{q}=\widehat{q}(k,\kappa;t)$ as follows.
\begin{lemma}\label{lemma:explicitlinearodddegree}
Let $\widehat{p}=\widehat{p}(k,\kappa;t)$ and 
$\widehat{q}=\widehat{q}(k,\kappa;t)$ denote the
two-dimensional Fourier transforms of the solutions to the 
linear base and auxiliary equations just above.
Assume that $q(x,y;0)=\delta(x-y)$ and $p(x,y;0)=p_0(x,y)$.
Then for all $t\geqslant0$ the functions $\widehat{p}$ and 
$\widehat{q}$ are explicitly given by
\begin{align*}
\widehat{p}(k,\kappa;t)&=\exp\bigl(-\mathrm{i}t\,h(2\pi\mathrm{i}k)\bigr)\,\widehat{p}_0(k,\kappa),\\
\widehat{q}(k,\kappa;t)&=\exp\bigl(-\mathrm{i}t\,h(-2\pi\mathrm{i}k)\bigr)\cdot
\exp^\star\Bigl(t\bigl(\widehat{f}^\star(\widehat{p}_0\star\widehat{p}_0^\dag)
+\mathrm{i}h\cdot\delta\bigr)\Bigr)(k,\kappa;t),
\end{align*}
where naturally $\exp^\star(c)=\delta+c+\frac12c^{\star2}+\frac16c^{\star3}+\cdots$.
\end{lemma}
\begin{proof}
The explicit form for $\widehat{p}=\widehat{p}(k,\kappa;t)$ follows directly from
the Fourier transformed version of the base equation. We now focus on the auxiliary equation. 
Consider a typical term say $\widehat{c}^{\star m}$
in $\widehat{f}^\star$, with $\widehat{c}\coloneqq\widehat{p}\star\widehat{p}^\dag$. 
Using Parseval's identity the term 
$\widehat{c}^{\star m}=\bigl(\widehat{p}\star\widehat{p}^\dag\bigr)^{\star m}$ has the explicit
form
\begin{equation*}
\widehat{c}^{\star m}(\nu_0,\nu_m;t)=\int_{\R^{2m-1}}
\biggl(\prod_{j=1}^m \widehat{p}(\nu_{j-1},\lambda_j;t)\widehat{p}^*(\nu_j,\lambda_j;t)\biggr)
\rd\lambda_1\cdots\rd\lambda_m\,\rd\nu_1\cdots\rd\nu_{m-1}. 
\end{equation*}
If we insert the explicit solution for $\widehat{p}$ into this expression and use that
$h$ is a polynomial of even degree terms only, we find
\begin{multline*}
\widehat{c}^{\star m}(\nu_0,\nu_m;t)=
\exp\Bigl(-\mathrm{i}t\bigl(h(2\pi\mathrm{i}\nu_0)-h(-2\pi\mathrm{i}\nu_m)\bigr)\Bigr)\\
\times\int_{\R^{2m-1}}
\biggl(\prod_{j=1}^m \widehat{p}_0(\nu_{j-1},\lambda_j)\widehat{p}_0^*(\nu_j,\lambda_j)\biggr)
\rd\lambda_1\cdots\rd\lambda_m\,\rd\nu_1\cdots\rd\nu_{m-1}. 
\end{multline*}
Hence we deduce that 
\begin{equation*}
\bigl(\widehat{f}^{\star}(\widehat{p}\star\widehat{p}^\dag)\bigr)(\nu_0,\nu_m;t)
=\exp\Bigl(-\mathrm{i}t\bigl(h(2\pi\mathrm{i}\nu_0)-h(-2\pi\mathrm{i}\nu_m)\bigr)\Bigr)
\bigl(\widehat{f}^{\star}(\widehat{p}_0\star\widehat{p}_0^\dag)\bigr)(\nu_0,\nu_m).
\end{equation*}
The auxiliary equation thus has the explicit form
\begin{equation*}
\pa_t\widehat{q}(k,\kappa;t)=\int_{\R}
\exp\Bigl(-\mathrm{i}t\bigl(h(2\pi\mathrm{i}k)-h(-2\pi\mathrm{i}\nu)\bigr)\Bigr)
\bigl(\widehat{f}^{\star}(\widehat{p}_0\star\widehat{p}_0^\dag)\bigr)(k,\nu)\,
\widehat{q}(\nu,\kappa;t)\,\rd\nu.
\end{equation*}
By making a change of variables we can convert this linear differential
equation for $\widehat{q}=\widehat{q}(k,\kappa;t)$ into a constant coefficient
linear differential equation. Indeed we set
\begin{equation*}
\widehat{\theta}(k,\kappa;t)
\coloneqq\exp\bigl(\mathrm{i}t\,h(-2\pi\mathrm{i}k)\bigr)\widehat{q}(k,\kappa;t).
\end{equation*}
Combining this definition with the linear differential
equation for $\widehat{q}=\widehat{q}(k,\kappa;t)$ above, we find
\begin{equation*}
\pa_t\widehat{\theta}(k,\kappa;t)
=\int_{\R}\bigl(\widehat{f}^{\star}(\widehat{p}_0\star\widehat{p}_0^\dag)\bigr)(k,\nu)\,
\widehat{\theta}(\nu,\kappa;t)\,\rd\nu
+\mathrm{i}\,h(-2\pi\mathrm{i}k)\widehat{\theta}(k,\kappa;t),
\end{equation*}
where, crucially, we again used that $h(-2\pi\mathrm{i}k)-h(2\pi\mathrm{i}k)\equiv0$
as $h$ is a polynomial of even degree terms. Hence 
the evolution equation for $\widehat{\theta}$ is the linear constant 
coefficient equation
\begin{equation*}
\pa_t\widehat{\theta}=\bigl(\widehat{f}^\star(\widehat{p}_0\star\widehat{p}_0^\dag)
+\mathrm{i}h\cdot\delta\bigr)\star\widehat{\theta},
\end{equation*}
Note the coefficient function depends only on the initial data $p_0$. 
Further note we have used that
\begin{equation*}
\bigl((\mathrm{i}h\,\delta)\star\widehat{\theta}\bigr)(k,\kappa;t)
=\mathrm{i}\,h(-2\pi\mathrm{i}k)
\int_\R\delta(k-\nu)\widehat{\theta}(\nu,\kappa;t)\,\rd\nu
=\mathrm{i}\,h(-2\pi\mathrm{i}k)\widehat{\theta}(k,\kappa;t).
\end{equation*}
Let us now focus on the initial data. Recall that we choose $q(x,y;0)=\delta(x-y)$
corresponding to $q^\prime(x,y;0)=0$. Hence we have 
$\widehat{q}(k,\kappa;0)=\widehat{\theta}(k,\kappa;0)=\delta(k-\kappa)$. 
The solution to the linear constant coefficient equation for 
$\widehat{\theta}=\widehat{\theta}(k,\kappa;t)$, by iteration, 
can thus be expressed in the form 
\begin{equation*}
\widehat{\theta}(k,\kappa;t)
=\exp^\star\Bigl(t\bigl(\widehat{f}^\star(\widehat{p}_0\star\widehat{p}_0^\dag)
+\mathrm{i}h\cdot\delta\bigr)\Bigr)(k,\kappa;t),
\end{equation*}
where $\exp^\star(c)=\delta+c+\frac12c^{\star2}+\frac16c^{\star3}+\cdots$.
We can recover $\widehat{q}$ from the definition for 
$\widehat{\theta}$ above. \qed
\end{proof}
\begin{remark}
The iterative procedure alluded to in the proof just above
ensures the correct interpretation of the terms in the exponential 
expansion $\exp^\star$ in the expression for $\widehat{q}=\widehat{q}(k,\kappa;t)$ 
above. Hence for example we have
$(\widehat{f}+\mathrm{i}h\cdot\delta)^{\star2}=
\widehat{f}\star\widehat{f}
+\widehat{f}\star(\mathrm{i}h\cdot\delta)
+\mathrm{i}h\cdot\widehat{f}
+(\mathrm{i}h)\cdot(\mathrm{i}h)\cdot\delta$.
\end{remark}
\begin{remark}[Hilbert--Schmidt solutions]\label{remark:HS2}
Here we suppose $\mathbb H=L^2(\mathbb R;\mathbb C)\times\mathrm{Dom}(D)$ 
with $\mathrm{Dom}(D)\subseteq L^2(\mathbb R;\mathbb C)$ and 
$\mathbb H=\mathbb Q\oplus\mathbb P$ with $\mathbb P$ 
and $\mathbb Q$ closed subspaces of $\Hb$; see Beck \textit{et al.\/} ~\cite{BDMS}.
The functions in $\Qb$ and $\Pb$ are both $\C$-valued. As in Remark~\ref{remark:HS},
with $Q(t)=\id+Q^\prime(t)$, the operators $Q^\prime(t)\in\mathfrak J_2(\Qb;\Qb)$ 
and $P(t)\in\mathfrak J_2(\Qb;\Pb)$ can be characterized, respectively, 
by kernel functions $q^\prime(\cdot,\cdot;t)\in L^2(\R^2;\C)$ and 
$p(\cdot,\cdot;t)\in L^2(\R^2;\C)$. Further we have the usual isometry
of Hilbert--Schmidt and $L^2(\R^2;\C)$-norms. The linear base and auxiliary 
equations for $p=p(x,y;t)$ and $q^\prime=q^\prime(x,y;t)$ are the versions
of the linear base and auxiliary equations in 
Definition~\ref{definition:odddegreebase} written in terms of their 
integral kernels; with $q(x,y;t)=\delta(x-y)+q^\prime(x,y;t)$.
Note we set $D=d(\pa_1)$ and indeed 
$d(\pa_1)=-\mathrm{i}\,h(\pa_1)$ where $h$ is a polynomial 
of even degree terms only with constant coefficients. Hence
$d=d(\pa_1)$ is of dispersive form and
$P\in C^\infty\bigl([0,T];\mathrm{Dom}(D)\bigr)$ as required
in the ``abstract'' formulation. We observe from the form
of the Fourier transform for the solution $\widehat{p}=\widehat{p}(k,\kappa;t)$
given in Lemma~\ref{lemma:explicitlinearodddegree}, that any 
Fourier Sobolev norm of the solution at any time $t>0$ equals 
the corresponding Fourier Sobolev norm of the initial data $\widehat{p}_0(k,\kappa)$.
Hence if the initial data is smooth, which we assume, so is $p=p(x,y;t)$ for all $t>0$.
Let us now focus on $q=q(x,y;t)$ which we recall satisfies the linear
auxiliary equation $\pa_t q=f^\star\bigl(p\star p^\dag\bigr)\star q$ and
the initial condition $q(x,y;0)=\delta(x-y)$.
Since $p=p(x,y;t)$ is bounded in any Sobolev norm for all $t>0$, 
so is $f^\star\bigl(p\star p^\dag\bigr)$. Let $\mathfrak p(t)$
denote the function $\{(x,y)\mapsto p(x,y;t)\}$, while 
$\mathfrak q(t)$ denotes the function $\{(x,y)\mapsto q(x,y;t)\}$
and $\mathfrak f(t)$ denotes the function $\{(x,y)\mapsto f^\star(x,y;t)\}$. 
By integrating in time, we can express the 
linear auxiliary equation in the abstract form 
\begin{equation*}
\mathfrak q(t)=\delta+\int_0^t\mathfrak f(\tau)\star\mathfrak q(\tau)\,\rd\tau.
\end{equation*}
Note we used that the Dirac delta function is the initial data, 
i.e.\/ $\mathfrak q(0)=\delta$. Recall it is also the unit with respect to
the `$\star$' product. We iterate this formula for $\mathfrak q(t)$
to generate the solution series
\begin{equation*}
\mathfrak q(t)=\delta+\int_0^t\mathfrak f(\tau)\rd\tau
+\int_0^t\int_0^{\tau}\mathfrak f(\tau)\star\mathfrak f(s)\,\rd s\,\rd\tau
+\int_0^t\int_0^{\tau}\int_0^{s}
\mathfrak f(\tau)\star\mathfrak f(s)\star\mathfrak f(r)
\,\rd r\,\rd s\,\rd\tau+\cdots.
\end{equation*}
Note that we have the following estimate for the $L^2(\R^2;\C)$-norm of
$\mathfrak f(\tau)\star\mathfrak f(s)$:
\begin{align*}
\bigl\|\mathfrak f(\tau)\star\mathfrak f(s)\bigr\|^2
&=\int_{\R^2}\biggl|\int_\R f(x,z;\tau)\,f(z,y;s)\,\rd z\biggr|^2\,\rd x\,\rd y\\
&\leqslant\int_{\R^2}\biggl(\int_\R |f|^2(x,z;\tau)\,\rd z\biggr)
\biggl(\int_\R |f|^2(z,y;s)\,\rd z\biggr)\,\rd x\,\rd y\\
&=\bigl\|\mathfrak f(\tau)\bigr\|^2\,\bigl\|\mathfrak f(s)\bigr\|^2.
\end{align*}
This estimate extends to 
$\bigl\|\mathfrak f(\tau)\star\mathfrak f(s)\star\mathfrak f(r)\bigr\|^2
\leqslant\bigl\|\mathfrak f(\tau)\bigr\|^2\,\bigl\|\mathfrak f(s)\bigr\|^2
\,\bigl\|\mathfrak f(r)\bigr\|^2$ and so forth. Since for any $T>0$ 
there exists a constant $K>0$ such that for all $t\in[0,T]$ we have
$\bigl\|\mathfrak f(t)\bigr\|^2\leqslant K$, we observe that 
the $L^2(\R^2;\C)$-norm of $\bigl(\mathfrak q(t)-\delta\bigr)$ is bounded 
as follows,
\begin{align*}
\bigl\|\mathfrak q(t)-\delta\bigr\|^2
&\leqslant\int_0^t\bigl\|\mathfrak f(\tau)\bigr\|^2\rd\tau
+\int_0^t\int_0^{\tau}\bigl\|\mathfrak f(\tau)\star\mathfrak f(s)\bigr\|^2\,\rd s\,\rd\tau
+\cdots\\
&\leqslant\int_0^t\bigl\|\mathfrak f(\tau)\bigr\|^2\rd\tau
+\int_0^t\int_0^{\tau}\bigl\|\mathfrak f(\tau)\bigr\|^2
\,\bigl\|\mathfrak f(s)\bigr\|^2\,\rd s\,\rd\tau
+\cdots\\
&\leqslant\exp(t\,K)-1.
\end{align*}
Consequently $\|Q^\prime(t)\|_{\mathfrak J_2(\Qb;\Qb)}$ is bounded. 
Further, recalling arguments in the proof of Lemma~\ref{lemma:EandU},
there exists a $T>0$ such that for all $t\in[0,T]$ we have 
$\|Q^\prime(t)\|_{\mathfrak J_2(\Qb;\Qb)}<1$ and 
$\mathrm{det}_2\bigl(\id+Q^\prime(t)\bigr)\neq0$. 
\end{remark}
\begin{corollary}[Evolutionary PDEs with odd degree nonlocal nonlinearities]
\label{cor:generalodd}
Given initial data $g_0\in C^\infty(\R^2;\C)\cap L^2(\R^2;\C)$, 
suppose $p=p(x,y;t)$ and $q=q(x,y;t)$
are the solutions to the linear base and auxiliary equations
from Lemma~\ref{lemma:explicitlinearodddegree} for which $p_0\equiv g_0$ 
and $q(x,y;0)=\delta(x-y)$. Let $\mathrm{Dom}(d)$ denote the domain
of the operator $d=-\mathrm{i}\,h(\pa_1)$ where $h=h(\pa_1)$ is defined above. 
Then there exists a $T>0$ such that the solution 
$g\in C^\infty\bigl([0,T];\mathrm{Dom}(d)\cap L^2(\R^2;\C)\bigr)$ 
to the linear Fredholm equation
\begin{equation*}
p(x,y;t)=\int_{\R}g(x,z;t)\,q(z,y;t)\,\rd z.
\end{equation*}
solves the evolutionary partial differential equation with 
odd degree nonlocal nonlinearity of the form
\begin{equation*}
\pa_tg=-\mathrm{i}h(\pa_1)\,g-g\star f^\star(g\star g^\dag).
\end{equation*}
\end{corollary}
\begin{proof}
From Remark~\ref{remark:HS2} we know that with a slight modification
of Lemma~\ref{lemma:EandU} for some $T>0$ there exists a solution 
$g\in C^\infty\bigl([0,T];\mathrm{Dom}(d)\cap L^2(\R^2;\C)\bigr)$ 
to the linear Fredholm equation (Riccati relation) shown. The solution $g$
is the integral kernel of $G$, which solves the Odd Degree Evolution
Equation in Theorem~\ref{theorem:generalmain}. Writing that 
equation in terms of the kernel function $g$ corresponds to the 
partial differential equation with odd degree nonlocal 
nonlinearity shown.\qed
\end{proof}
\begin{remark} We make the following observations: 
(i) Though we have a closed form for $p=p(x,y;t)$ in this case,
$q=q(x,y;t)$ has a series representation. However as for our results
in \S\ref{sec:quadflows}, time $t$ plays the role of a parameter in
the sense that we decide on the time at which we wish to evaluate the solution,
and then we solve the Fredholm equation to generate the solution $g$ for that time $t$;   
(ii) Also as for our results in \S\ref{sec:quadflows}, on the interval of
time for which we know $g$ exists, its regularity is determined by the 
regularity of $p$; and 
(iii) The extension of our results above to the case when $p$, $q$ and $g$
are $\C^{n\times n}$-valued functions for any $n\in\mathbb N$ is straightforward.
\end{remark}
There are many generalizations and concomitant results we intend to pursue. 
A few immediate ones are as follows. In all cases we assume the base equation 
to be $\pa_tP=DP$ and the Riccati relation has the form $P=G\,Q$. 
First, in the nonlocal cubic case assume the 
auxiliary equation has the form $\pa_tQ=(PAP^\dag)\,Q$
for some linear operator $A$ satisfying $A^\dag=-A$. This generates the cubic
form of the operator equation for $G$ in 
the Odd Degree Evolution Equation Theorem~\ref{theorem:generalmain} above.
However we observe $\pa_t(QAQ^\dag)=[PAP^\dag,QAQ^\dag]$. Hence if the 
commutator on the right vanishes initially then $QAQ^\dag$ maintains its
initial value thereafter. If we assume $Q_0AQ_0^\dag=\mathrm{i}\alpha\cdot\id$ 
then we recover the same result as that in Theorem~\ref{theorem:generalmain}
with the scalar $\alpha$ forced to be real from the skew-Hermitian property of $A$.
Second, suppose the auxiliary equation has the form $\pa_tQ=(A_1PA_2P^\dag A_3)\,Q$
for some operators $A_1$, $A_2$ and $A_3$. Assuming $Q$ satisfies the constraint
$QA_2Q^\dag=K$ for some time independent operator $K$ then $G$ can be shown to
satisfy $\pa_tG=D\,G-G\,(A_1GKGA_3)$. However, if $A_2^\dag=-A_2$ and $A_3=\pm A_1^\dag$,
then we observe that $\pa_t(QAQ^\dag)=\pm[A_1PA_2P^\dag A_1,K]$. Hence similarly,
if the commutator on the right vanishes initially and $Q_0A_2Q_0^\dag=K$ initially,
then this constraint is maintained thereafter. Third and lastly, we observe we
could assume the auxiliary equation has the form $\pa_tQ=f(PP^\dag)\,P$ to attempt
to generate even degree equations. We address further generalizations 
in our Conclusion \S\ref{sec:conclu}.

\section{Examples}\label{sec:examples}
We consider six example evolutionary partial differential equations 
with nonlocal nonlinearities in detail.
The first four examples are: (i) A reaction-diffusion system 
with nonlocal nonlinear reaction terms; 
(ii) The nonlocal Korteweg de Vries equation; (iii) A nonlocal nonlinear 
Schr\"odinger equation and (iv) A fourth order nonlinear Schr\"odinger equation
with a nonlocal sinusoidal nonlinearity. In each of these cases we 
provide the following. First, we present the evolutionary system and 
initial data and explain how it fits into the context of one 
of the systems presented in \S\ref{sec:quadflows} or \S\ref{sec:generalflows}. 
Second, we briefly explain how we simulated the evolutionary system
with nonlocal nonlinearity directly by adapting well-known algorithms,
mainly pseudo-spectral, for the versions of these systems with 
local nonlinearities. We denote these directly computed solutions 
by $g_{\mathrm{D}}$. Third, we explain in some more detail how
we generated solutions from the underlying linear base and 
auxiliary equations and the linear Riccati relation. We denote 
solutions computed using our Riccati method by $g_{\mathrm{R}}$. 
Then for a particular evaluation time $T>0$ we compute $g_{\mathrm{D}}$
and $g_{\mathrm{R}}$. We compare the two simulation results and 
explicitly plot their difference at that time $T$. We also
quote a value for the maximum norm over the spatial domain 
of the difference $g_{\mathrm{D}}-g_{\mathrm{R}}$. Additionally
we plot the evolution of $\mathrm{det}_2\bigl(\id+Q^\prime(t)\bigr)$, 
and in the first two examples $\|Q^\prime(t)\|_{\mathfrak J_2(\Qb;\Qb)}$. 
We emphasize that for all the examples, to compute $g_{\mathrm{R}}$ we simply 
evaluate the explicit forms for $p=p(x,y;t)$ and $q^\prime=q^\prime(x,y;t)$ 
or their Fourier transforms at the given time $t=T$. We then solve the 
corresponding Fredholm equation at time $t=T$ to generate $g_{\mathrm{R}}$.
The evolution plots for $\mathrm{det}_2\bigl(\id+Q^\prime(t)\bigr)$ and 
$\|Q^\prime(t)\|_{\mathfrak J_2(\Qb;\Qb)}$ are provided for interest
and analysis only. We remark that in some examples, at the 
evaluation times $t=T$, the norm $\|Q^\prime(t)\|_{\mathfrak J_2(\Qb;\Qb)}$ 
is greater than one. This suggests that the estimates
in Lemma~\ref{lemma:EandU}, whilst guaranteeing the behaviour required, 
are somewhat conservative. All the simulations are developed on the 
domain $[-L/2,L/2]^2$ with the problem projected spatially onto $M^2$ nodes, 
i.e.\/ $M$ nodes for the $x\in[-L/2,L/2]$ interval and 
$M$ nodes for the $y\in[-L/2,L/2]$ interval. Naturally $M^2$
also represents the number of two-dimensional Fourier modes
in our simulations. In each case we quote $L$ and $M$. All our
Matlab codes are provided in the supplementary electronic material.

The last two examples we present represent interesting special cases of our Riccati approach.
They are a: (v) Scalar evolutionary diffusive partial differential equation with a convolutional
nonlinearity and (vi) Nonlocal Fisher--Kolmogorov--Petrovskii--Piskunov equation
from biology/ecological systems. In the latter case we derive solutions for general
initial data constructed using our approach. As far as we know these have not been 
derived before.

\begin{example}[Reaction-diffusion system with nonlocal reaction terms]\label{ex:RDE}
In this case the target equation is the system of reaction-diffusion equations
with nonlocal reaction terms of the form
\begin{align*}
\pa_tu&=d_{11}u+d_{12}v-u\star(b_{11}u)-u\star(b_{12}v)-v\star(b_{12}u)-v\star(b_{11}v),\\
\pa_tv&=d_{11}v+d_{12}u-u\star(b_{11}v)-u\star(b_{12}u)-v\star(b_{12}v)-v\star(b_{11}u),
\end{align*}
where $u=u(x,y;t)$ and $v=v(x,y;t)$. We assume $d_{11}=\pa_1^2+1$, 
$d_{12}=-1/2$, $b_{12}=0$ and $b_{11}=N(x,\sigma)$, 
the Gaussian probability density function with mean zero. We set $\sigma=0.1$.
We take the initial profiles 
$u_0(x,y)\coloneqq\mathrm{sech}(x+y)\,\mathrm{sech}(y)$ and
$v_0(x,y)\coloneqq\mathrm{sech}(x+y)\,\mathrm{sech}(x)$,
and in this case  $L=20$ and $M=2^7$. This system fits into our general theory 
in \S\ref{sec:quadflows} when we take $p$, $q$ and $g$ 
to have the $2\times 2$ bisymmetric forms
\begin{equation*}
p=\begin{pmatrix}p_{11} & p_{12}\\ p_{12} & p_{11}\end{pmatrix},
\qquad
q=\begin{pmatrix}q_{11} & q_{12}\\ q_{12} & q_{11}\end{pmatrix}
\qquad\text{and}\qquad
g=\begin{pmatrix}g_{11} & g_{12}\\ g_{12} & g_{11}\end{pmatrix}.
\end{equation*}
We also assume similar forms for $d$ and $b$ with the components indicated above.
Note that the product of two $2\times2$ bisymmetric matrices is bisymmetric.
The resulting evolutionary Riccati equation $\pa_tG=dG-G\,(b G)$ in terms of the kernel
functions $g_{11}=u$ and $g_{12}=v$ is the target reaction-diffusion system 
with nonlocal nonlinearities above.
The results of our simulations are shown in Figure~\ref{fig:RDE}. 
The top two panels show the $u$ and $v$ components of the solution computed up until
time $T=0.5$ using a direct spectral integration approach. By this we mean we
solved the system of equations in Fourier space for $\widehat{u}=\widehat{u}(k,\kappa;t)$
and $\widehat{v}=\widehat{v}(k,\kappa;t)$. We used the Matlab inbuilt integrator
\texttt{ode45} to integrate in time. The middle two panels 
show the $g_{11}$ and $g_{12}$ components of the solution
computed using our Riccati approach which respectively correspond to 
$u$ and $v$. To generate the solutions $g_{11}$ and $g_{12}$ 
we solved the $2\times 2$ matrix
Fredholm equation for $g$ computing $p$ and $q$ as $2\times 2$ matrices directly
from their explicit Fourier transforms. We approximated the integral in the 
Fredholm equation using a simple Riemann rule and used the inbuilt Matlab Gaussian
elimination solver to find the solution. The bottom left panel shows the Euclidean 
norm of the difference $(u-g_{11},v-g_{12})$ for all $(x,y)\in[-L/2,L/2]^2$ 
at time $t=T$. The solutions numerically coincide and indeed for that time
$t=T$ we have $\|(u-g_{11},v-g_{12})\|_{L^\infty(\R^2;\R)}=3.6178\times10^{-5}$.
We also computed the mean values of $|u-g_{11}|$ and $|v-g_{12}|$
over the domain which are, respectively, $8.7796\times10^{-8}$
and $1.6967\times10^{-7}$. The bottom right panel shows the evolution of  
$\mathrm{det}_2(\id+Q^\prime(t))$ and also
$\|Q^\prime(t)\|_{\mathfrak J_2(\Qb;\Qb)}$ for $t\in[0,T]$.
\end{example}

\begin{figure}
  \begin{center}
  \includegraphics[width=5cm,height=5cm]{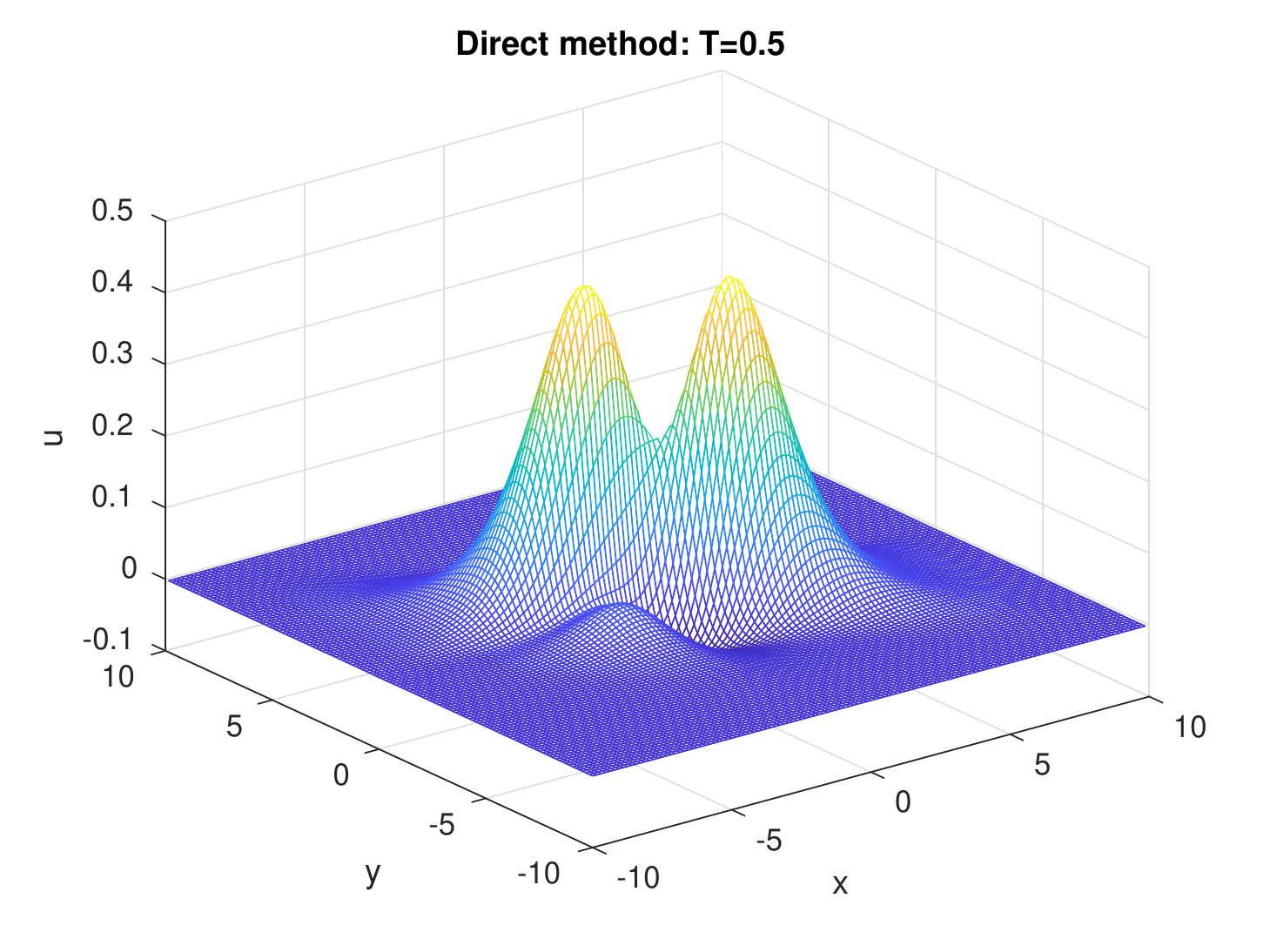}
  \includegraphics[width=5cm,height=5cm]{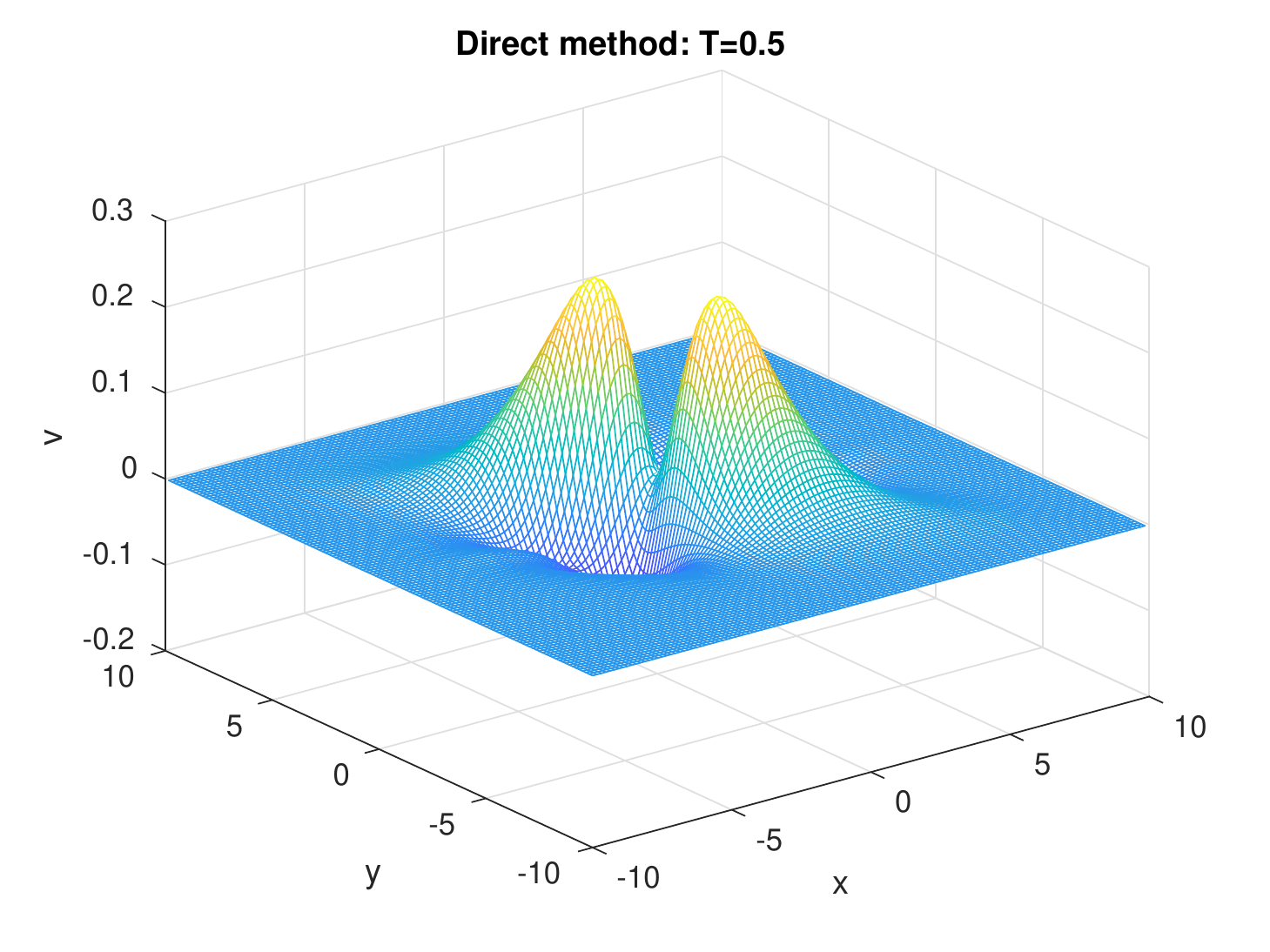}\\
  \includegraphics[width=5cm,height=5cm]{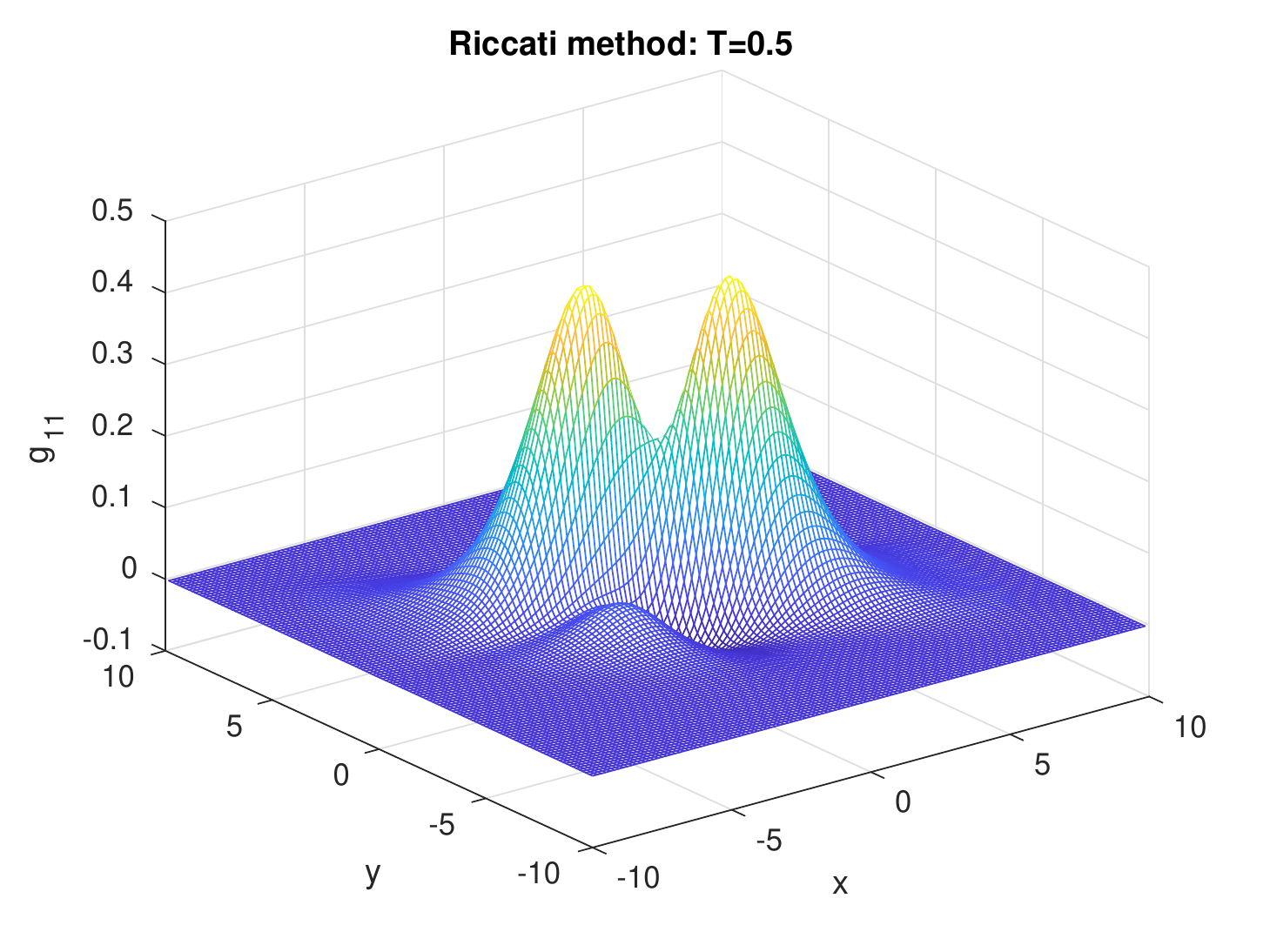}
  \includegraphics[width=5cm,height=5cm]{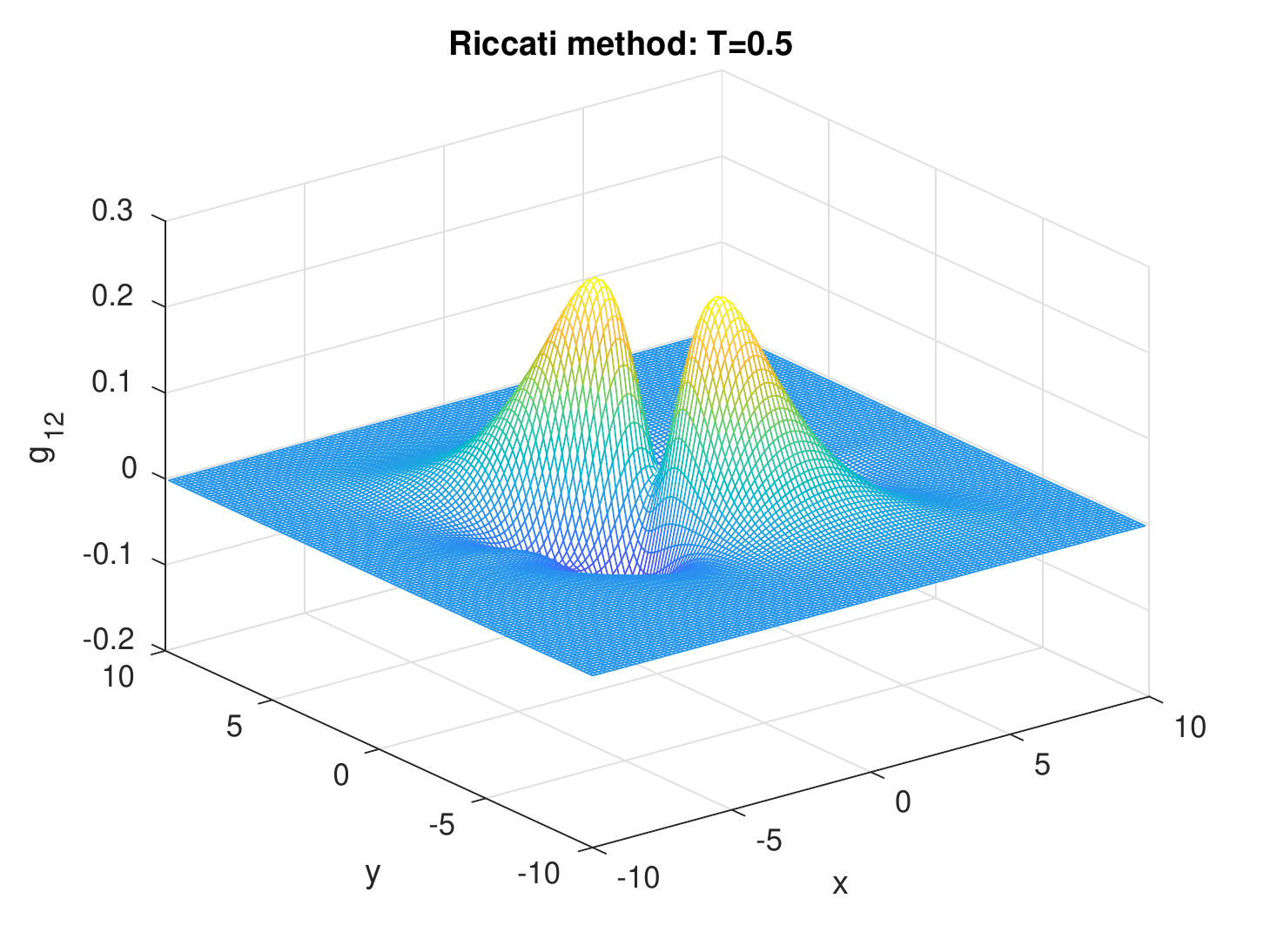}\\
  \includegraphics[width=5cm,height=5cm]{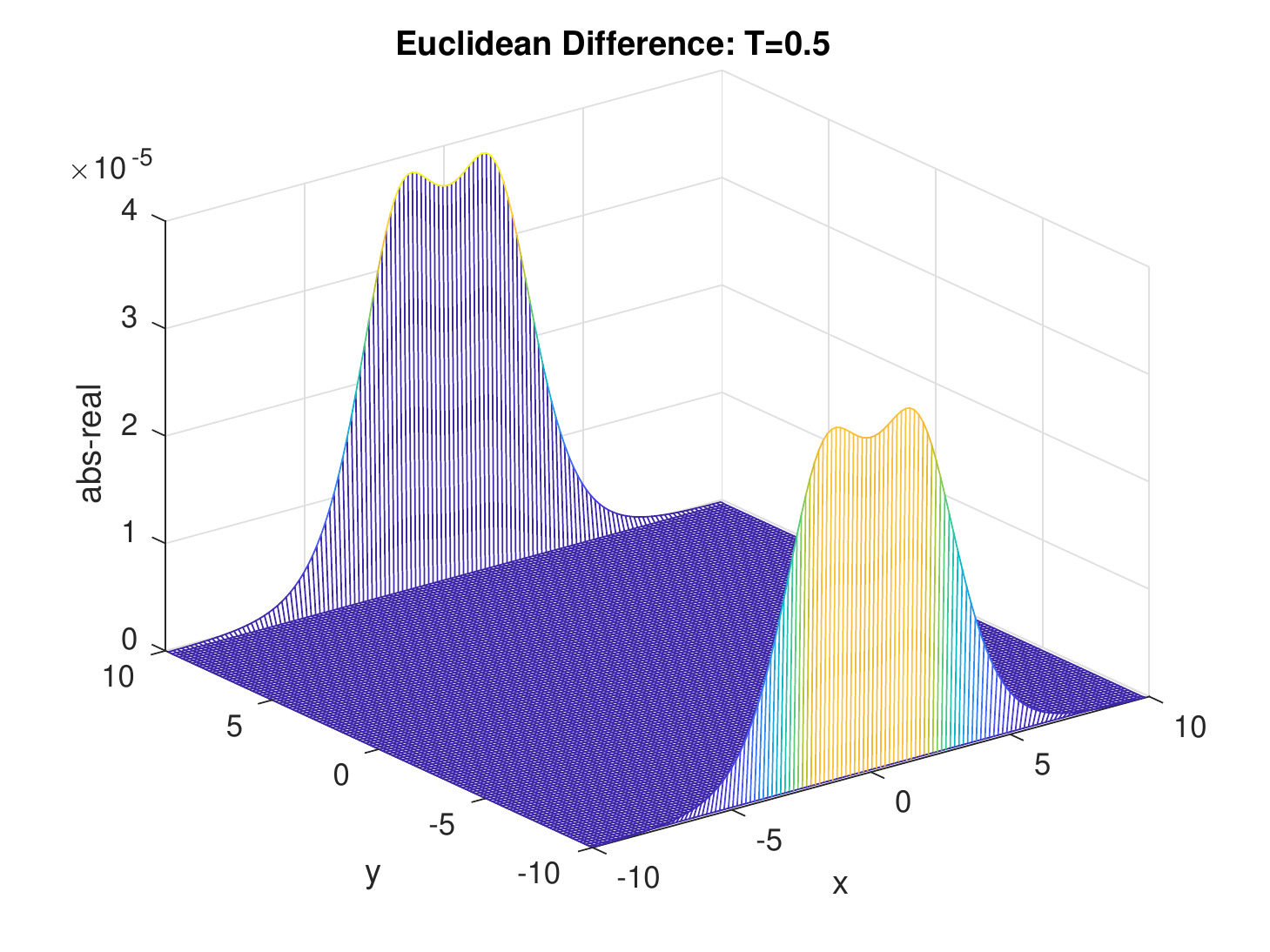}
  \includegraphics[width=5cm,height=5cm]{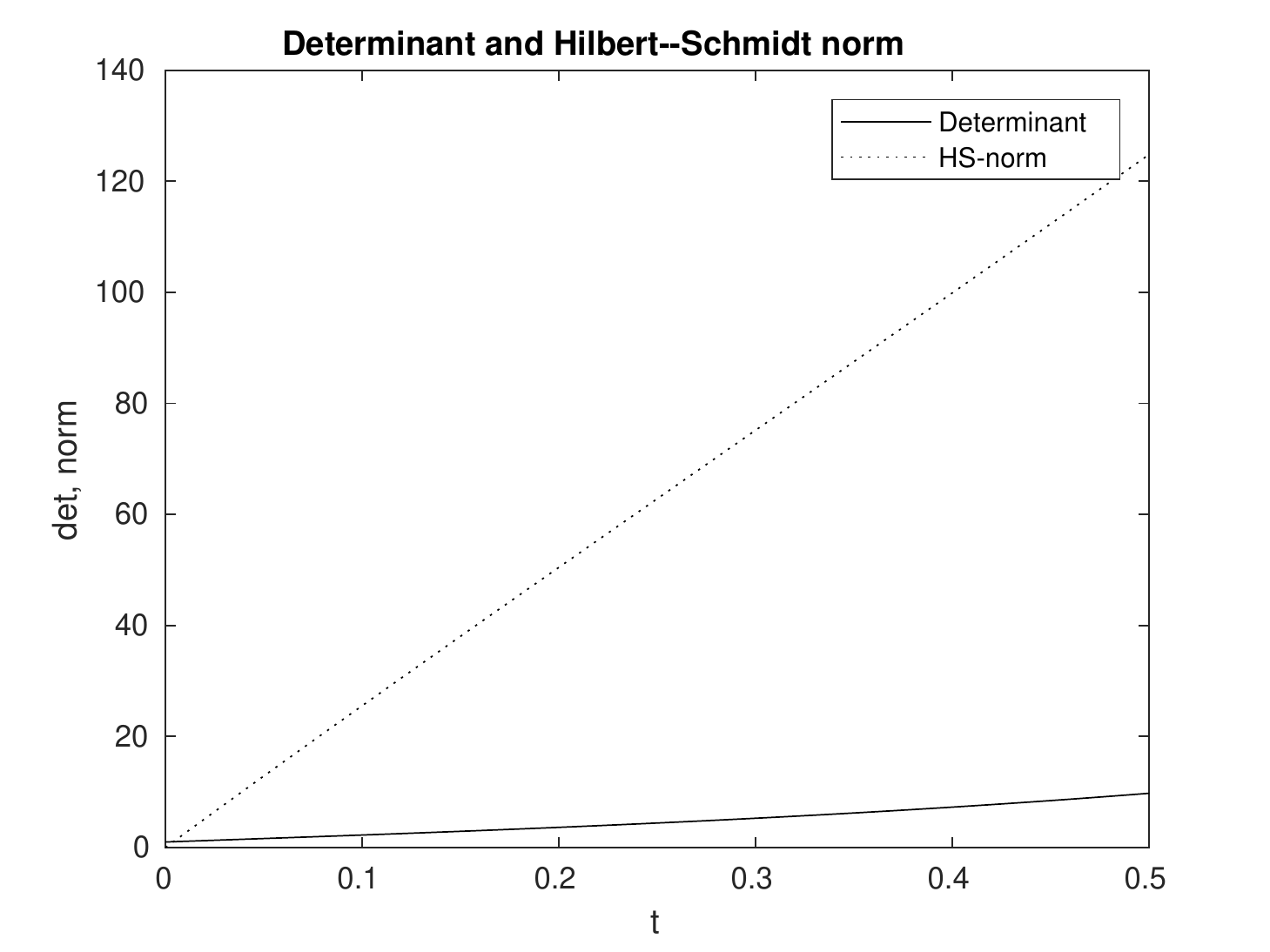}\\
  \end{center}
  \caption{We plot the solution to the nonlocal reaction-diffusion system
from Example~\ref{ex:RDE}. We used generic initial profiles 
$u_0(x,y)\coloneqq\mathrm{sech}(x+y)\,\mathrm{sech}(y)$ and
$v_0(x,y)\coloneqq\mathrm{sech}(x+y)\,\mathrm{sech}(x)$.
For time $T=0.5$, the top panels show the $u$ and $v$ components 
of the solution computed using a direct integration approach while the middle
panels show the corresponding $g_{11}$ and $g_{12}$ components 
of the solution computed using our Riccati approach.
The bottom left panel shows the Euclidean norm of 
the difference $(u-g_{11},v-g_{12})$ for all $(x,y)\in[-L/2,L/2]^2$.
The bottom right panel shows the evolution of the Fredholm Determinant 
and Hilbert--Schmidt norm associated with $Q^\prime(t)$ for $t\in[0,T]$.}
\label{fig:RDE}
\end{figure}

\begin{example}[Nonlocal Korteweg de Vries equation]\label{ex:KdV}
In this case the target equation is the nonlocal Korteweg de Vries equation
\begin{equation*}
\pa_tg=-\pa_1^3\,g-g\star (\pa_1 g),
\end{equation*}
for $g=g(x,y;t)$. Using our analysis in \S\ref{sec:quadflows} 
we thus need to set $d=-\pa_1^3$ and $b=\pa_1$. We choose an 
initial profile of the form
$g_0(x,y)\coloneqq \mathrm{sech}^2(x+y)\,\mathrm{sech}^2(y)$
and in this case  $L=40$ and $M=2^8$.
The results are shown in Figure~\ref{fig:KdV}. The top left panel
shows the solution  $g_{\mathrm{D}}$ computed up until time $T=1$ using a 
direct integration approach. By this we mean we
implemented a split-step Fourier Spectral approach modified
to deal with the nonlocal nonlinearity; we adapted the code from
that found at the Wikiwaves webpage~\cite{wikiwaves}. With the initial matrix
$\widehat{u}_0\coloneqq\widehat{g}_0$, indexed by the wavenumbers
$k$ and $\kappa$, the method is given by
(here $\mathcal{F}$ denotes the Fourier transform),
\begin{equation*}
\widehat{v}_{n}\coloneqq\exp\bigl(\Delta t\,K^3\bigr)\,\widehat{u}_n
\quad\text{and}\quad
\widehat{u}_{n+1}\coloneqq\widehat{v}_{n}+\Delta t\,h\,
\mathcal{F}\Bigl(\bigl(\mathcal{F}^{-1}(\widehat{v}_{n})\bigr)
\bigl(\mathcal{F}^{-1}(K\widehat{v}_{n})\bigr)\Bigr),
\end{equation*}
where $K$ is the diagonal matrix of Fourier coefficients $2\pi\mathrm{i}k$ and 
where the product between the two inverse Fourier transforms shown is the 
matrix product. In practice of course we used the fast Fourier transform.
Note we have chosen to approximate the nonlocal nonlinear term using a Riemann rule.
Further we used the time step $\Delta t=0.0001$.
The top right panel shows the solution $g_{\mathrm{R}}$ computed using our Riccati approach. 
By this we mean the following. We compute the explicit solutions for 
the base and auxiliary equations in this case in Fourier space in the form
\begin{equation*}
\widehat{p}(k,\kappa;t)=\mathrm{e}^{t(2\pi\mathrm{i}k)^3}\,\widehat{g}_0(k,\kappa)
\quad\text{and}\quad
\widehat{q}^\prime(k,\kappa;t)=(2\pi\mathrm{i}k)
\frac{\bigl(\mathrm{e}^{t(2\pi\mathrm{i}k)^3}-1\bigr)}{(2\pi\mathrm{i}k)^3}\,
\widehat{g}_0(k,\kappa).
\end{equation*}
Recall $q^\prime$ is the kernel associated with $Q^\prime=Q-\id$. After computing
the inverse Fourier transforms of these expressions we then solved the 
Fredholm equation, i.e.\/ the Riccati relation, for $\widehat{g}=\widehat{g}(x,y;t)$ 
numerically. There are three sources of error in this computation. The first is 
the wavenumber cut-off and inverse fast Fourier transform required to compute
$p=p(x,y;t)$ and $q'=q'(x,y;t)$ respectively from $\widehat{p}$ 
and $\widehat{q}^\prime$ above.
The second is in the choice of integral approximation in the Fredholm equation. 
We used a simple Riemann rule. The third is the error in solving the corresponding
matrix equation representing the Fredholm equation which is 
that corresponding to the error for Matlab's inbuilt Gaussian 
elimination solver. The bottom left panel shows the absolute value of 
$g_{\mathrm{D}}-g_{\mathrm{R}}$. Up to computation error,
the solutions naturally coincide, and indeed 
$\|g_{\mathrm{D}}-g_{\mathrm{R}}\|_{L^\infty(\R^2;\R)}=4.8871\times 10^{-5}$.
The bottom right panel shows the evolution of  
$\mathrm{det}_2(\id+Q^\prime(t))$ and also
$\|Q^\prime(t)\|_{\mathfrak J_2(\Qb;\Qb)}$ for $t\in[0,T]$.
\end{example}

\begin{figure}
  \begin{center}
  \includegraphics[width=5cm,height=5cm]{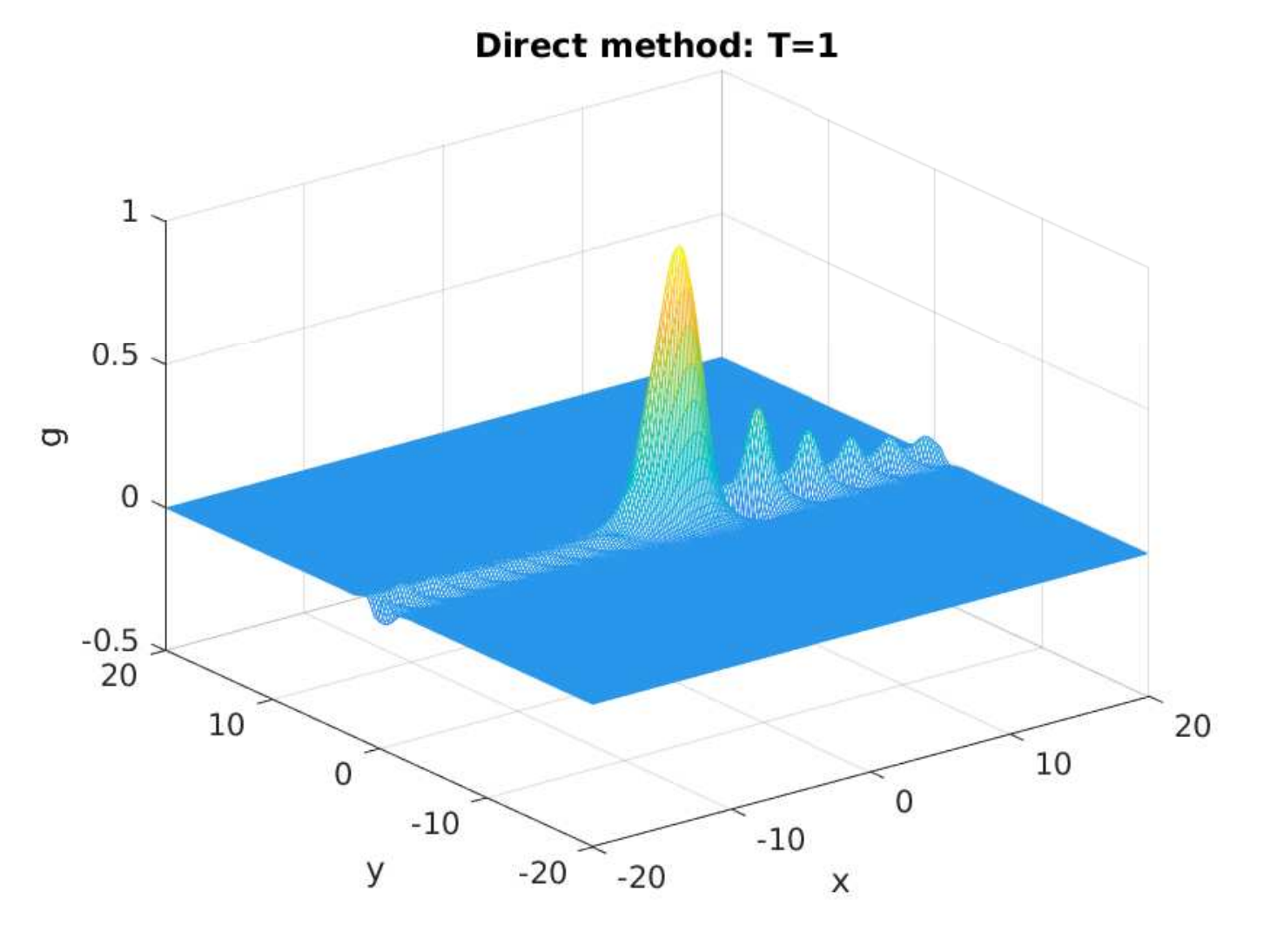}
  \includegraphics[width=5cm,height=5cm]{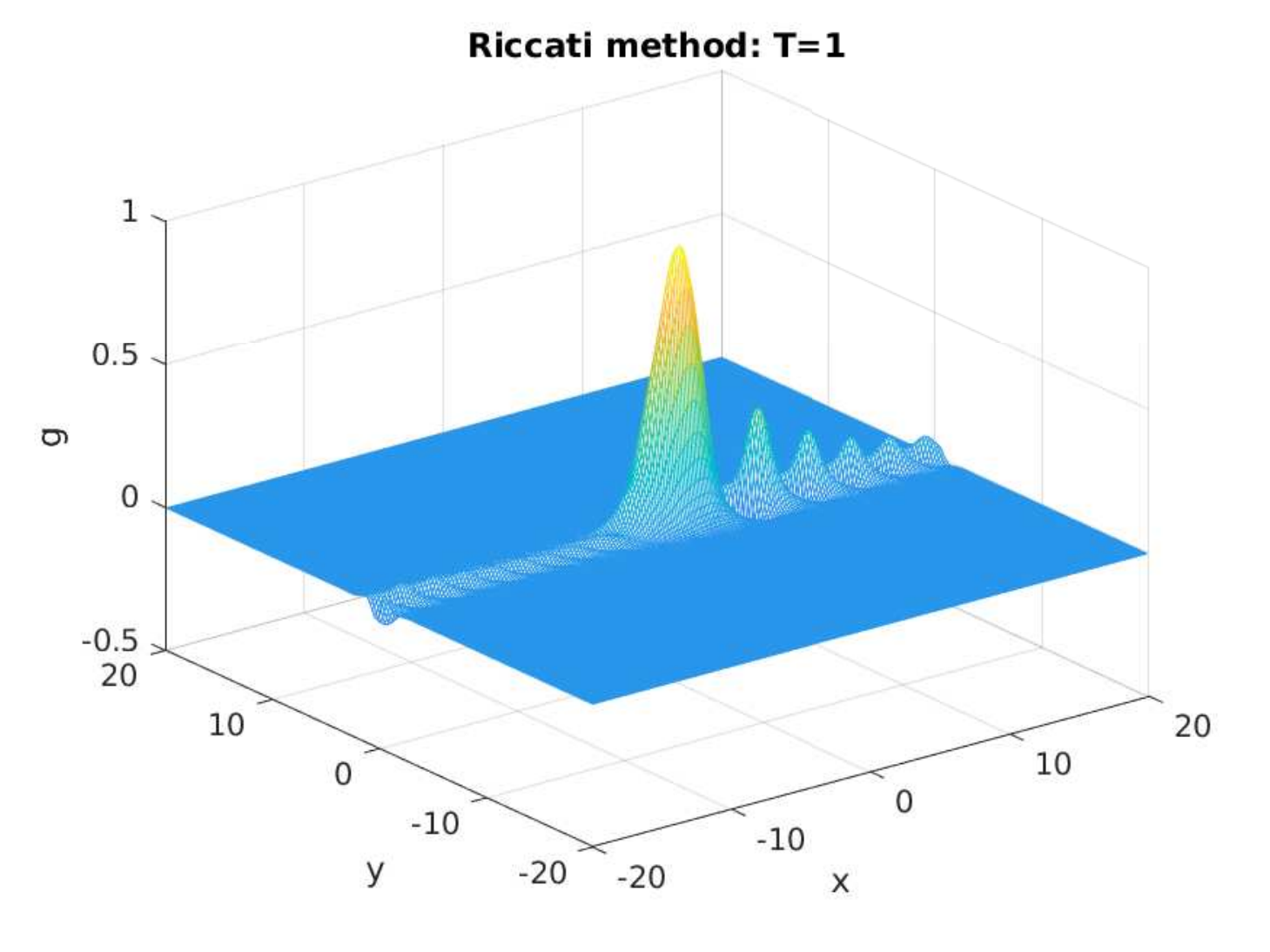}\\
  \includegraphics[width=5cm,height=5cm]{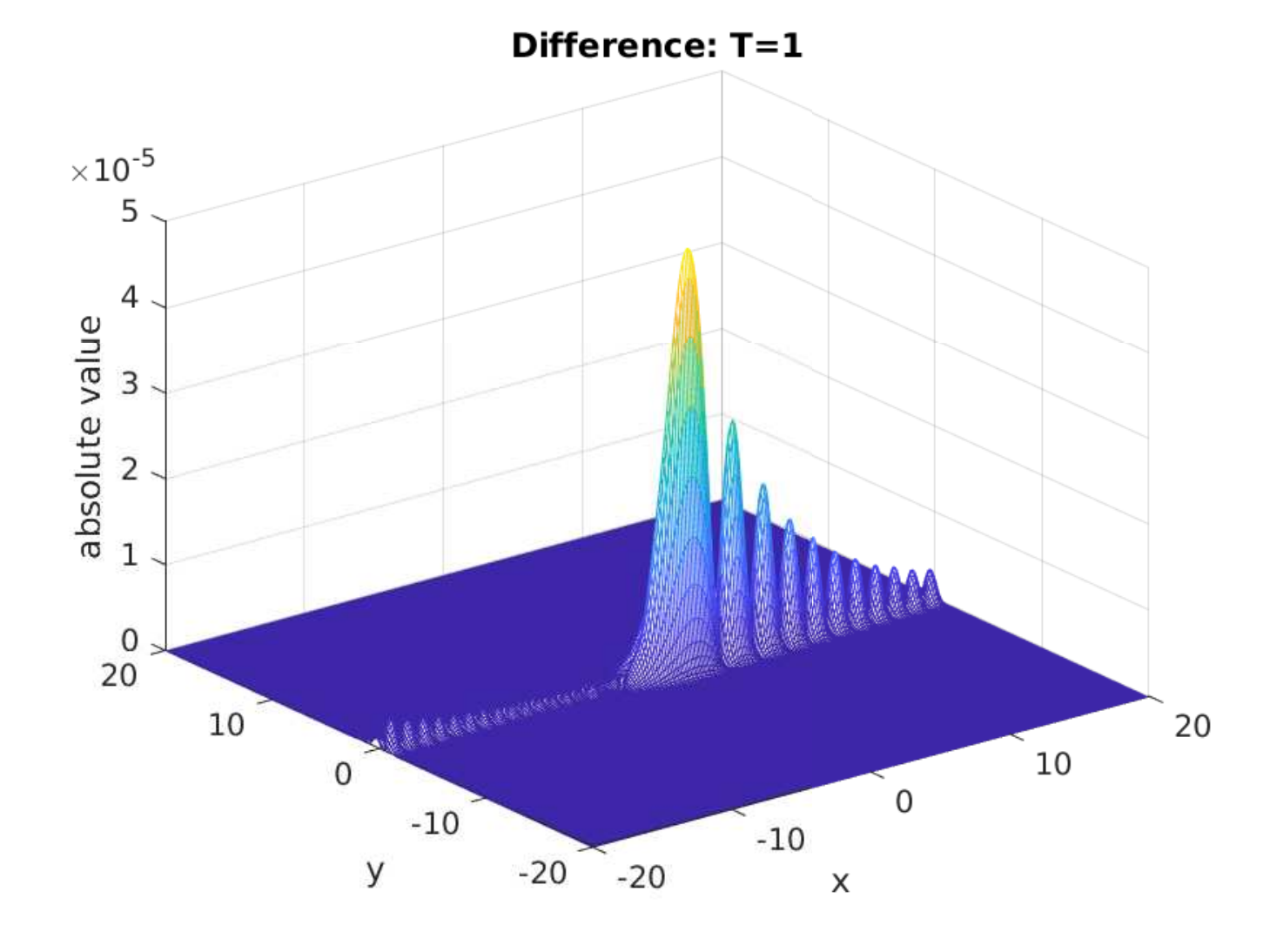}
  \includegraphics[width=5cm,height=5cm]{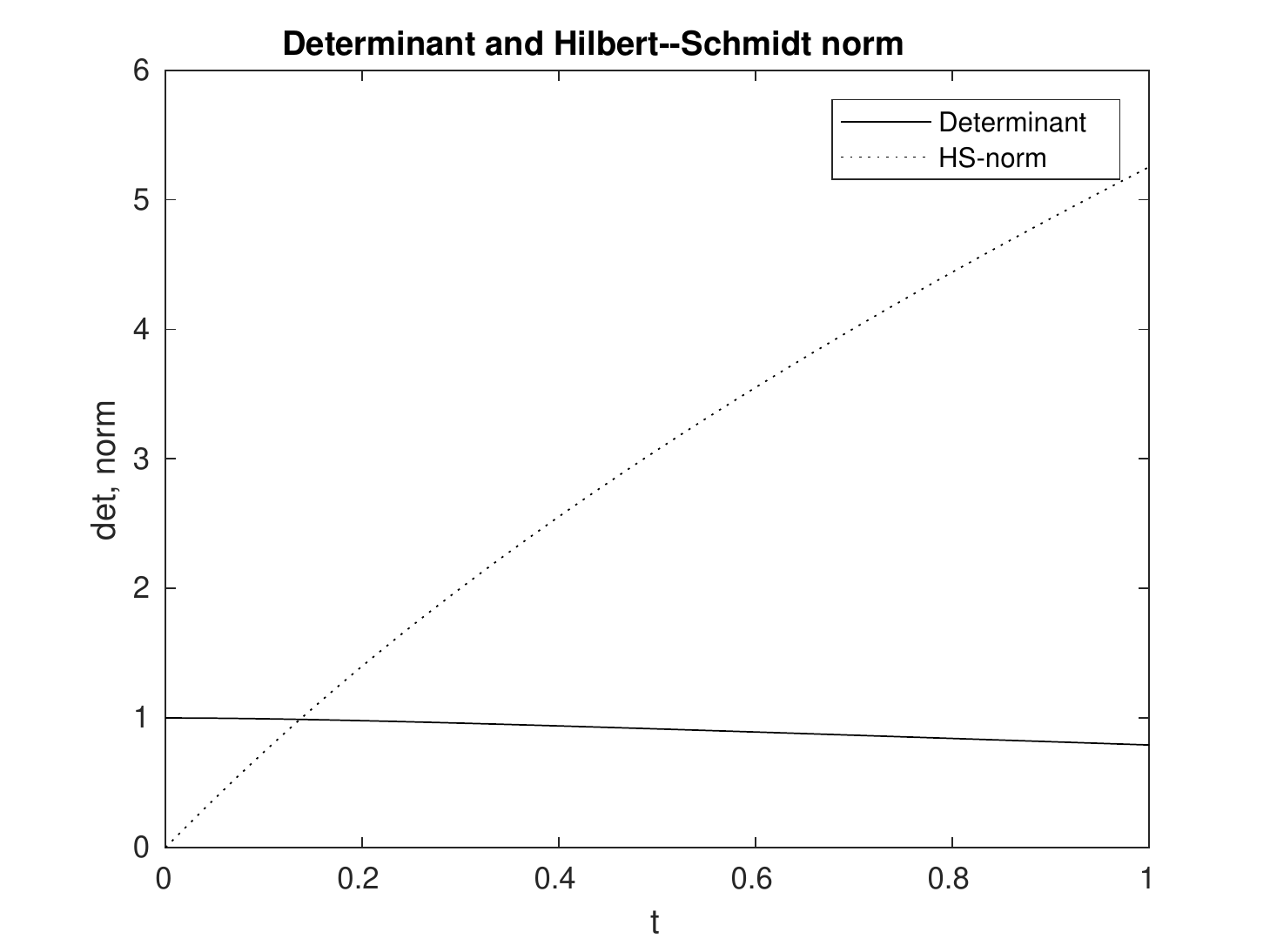}
  \end{center}
  \caption{We plot the solution to the nonlocal Korteweg de Vries equation
from Example~\ref{ex:KdV}. We used the generic initial profile 
$g_0(x,y)\coloneqq\mathrm{sech}^2(x+y)\,\mathrm{sech}^2(y)$.
For time $T=1$, the top panels show the solution computed using a 
direct integration approach (left) and the corresponding solution 
computed using our Riccati approach (right). 
The bottom left panel shows the absolute value of the difference
of the two computed solutions. The bottom right panel shows the 
evolution of the Fredholm Determinant and Hilbert--Schmidt norm
associated with $Q^\prime(t)$ for $t\in[0,T]$.}
\label{fig:KdV}
\end{figure}

\begin{example}[Nonlocal nonlinear Schr\"odinger equation]\label{ex:NLS}
In this case the target equation is the nonlocal nonlinear Sch\"odinger equation
\begin{equation*}
\mathrm{i}\pa_tg=\pa_1^2\,g+g\star g\star g^\dag,
\end{equation*}
for $g=g(x,y;t)$. In our analysis in \S\ref{sec:generalflows} 
we thus need to set $h(x)=x^2$ and $f(x)=x$.
Further for computations we take the initial profile 
to be $g_0(x,y)\coloneqq \mathrm{sech}(x+y)\,\mathrm{sech}(y)$
and in this case $L=20$ and $M=2^8$.
The results are shown in Figure~\ref{fig:NLS}. The top two panels
show the real and imaginary parts of the solution $g_{\mathrm{D}}$ computed up until
time $T=0.02$ using a direct integration approach. By this we mean we
implemented a split-step Fourier transform approach slightly modified
to deal with the nonlocal nonlinearity; see 
Dutykh, Chhay and Fedele~\cite[p.~225]{DCF}. The middle two panels 
show the real and imaginary parts of the solution $g_{\mathrm{R}}$ computed using our
Riccati approach. By this we mean, given the explicit solution for
$\widehat{p}=\widehat{p}(k,\kappa;t)$ in terms of $\widehat{g}_0$, 
we numerically evaluated $\widehat{q}=\widehat{q}(k,\kappa;t)$
using the exponential form from Lemma~\ref{lemma:explicitlinearodddegree}. 
In practice this consists of computing a large matrix exponential, 
our first source of error. We then solved the the Riccati relation 
in Fourier space when it takes the form 
$\widehat{p}(k,\kappa;t)=\int_{\R}\widehat{g}(k,\nu;t)\,\widehat{q}(\nu,\kappa;t)\,\rd\nu$
for $\widehat{g}=\widehat{g}(k,\kappa;t)$. We solved this Fredholm
equation numerically and recovered $g=g(x,y;t)$ 
as the inverse Fourier transform of $\widehat{g}=\widehat{g}(k,\kappa;t)$.
There are three further sources of error in this computation. 
The first is in the choice of integral approximation on the right-hand side. 
We used a simple Riemann rule. The second is the error in solving the corresponding
matrix equation representing the Fredholm equation which is 
that corresponding to the error for Matlab's in build Gaussian 
elimination solver. The third is in computing the inverse
fast Fourier transform for the solution. The bottom left panel shows 
$|g_{\mathrm{D}}-g_{\mathrm{R}}|$ for all $(x,y)\in[-L/2,L/2]^2$ at time $t=T$. 
Up to computation error, the solutions coincide, and we have  
$\|g_{\mathrm{D}}-g_{\mathrm{R}}\|_{L^\infty(\R^2;\C)}=2.6932\times 10^{-5}$.
The bottom right panel shows the evolution of  
$\mathrm{det}_2\bigl(\widehat{Q}(t)\bigr)$ for $t\in[0,T]$, i.e.\/ the 
Fredholm determinant of the Fourier transform $\widehat{q}$ of the 
kernel $q$ associated with $Q$. Not too surprisingly we observe 
$|\mathrm{det}_2\bigl(\widehat{Q}(t)\bigr)|=1$ for all $t\in[0,T]$. 
\end{example}

\begin{figure}
  \begin{center}
  \includegraphics[width=5cm,height=5cm]{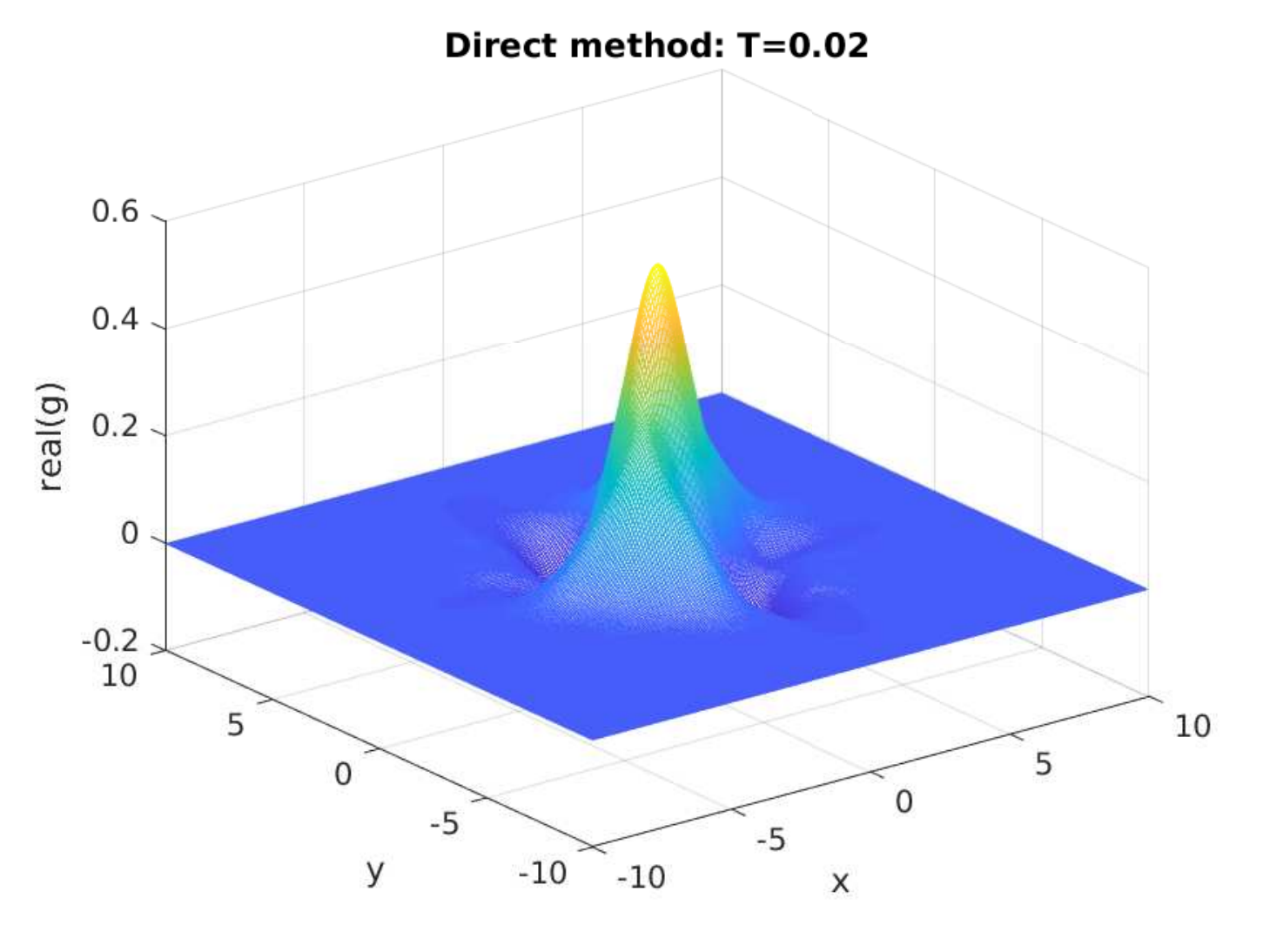}
  \includegraphics[width=5cm,height=5cm]{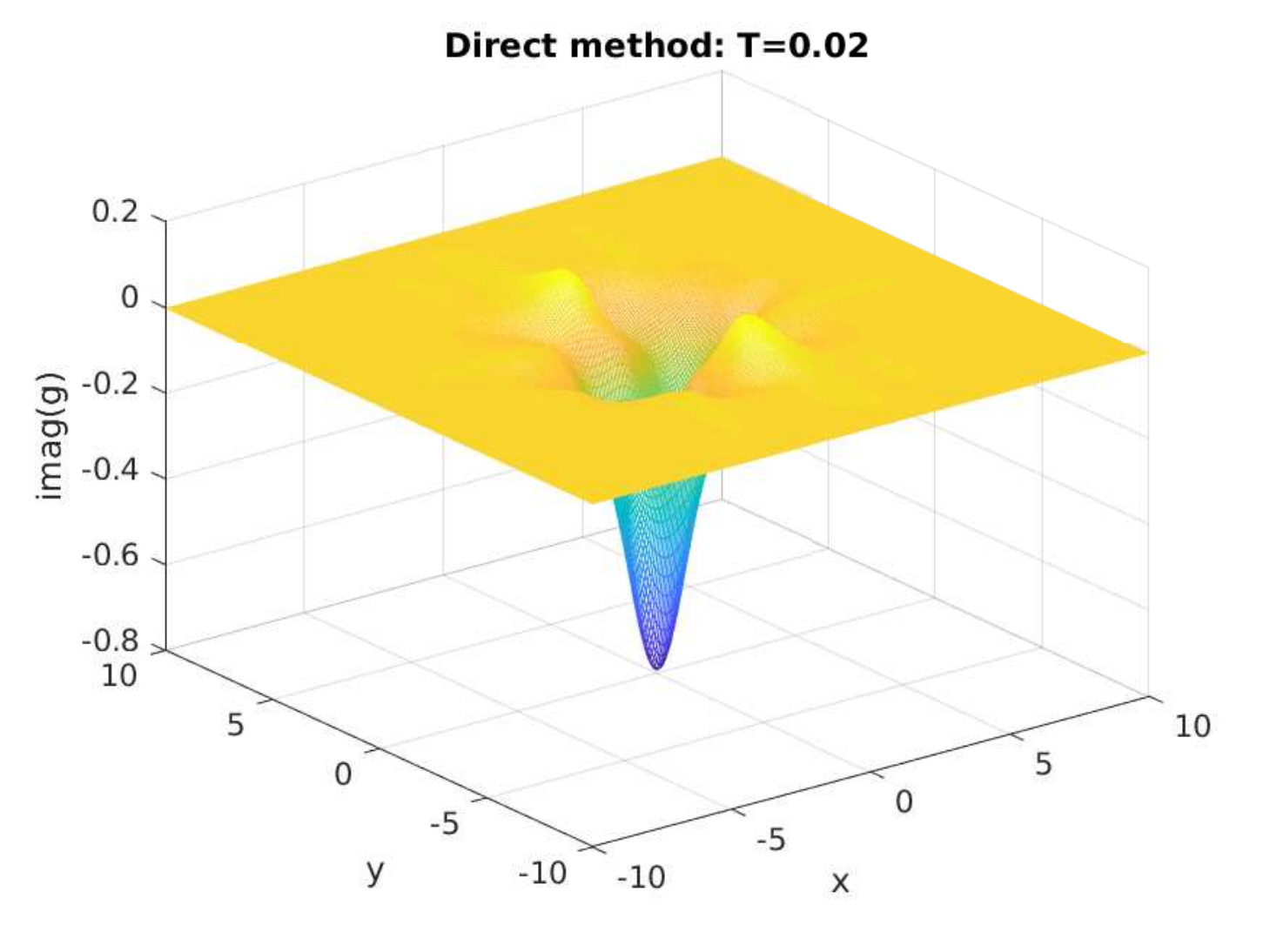}\\
  \includegraphics[width=5cm,height=5cm]{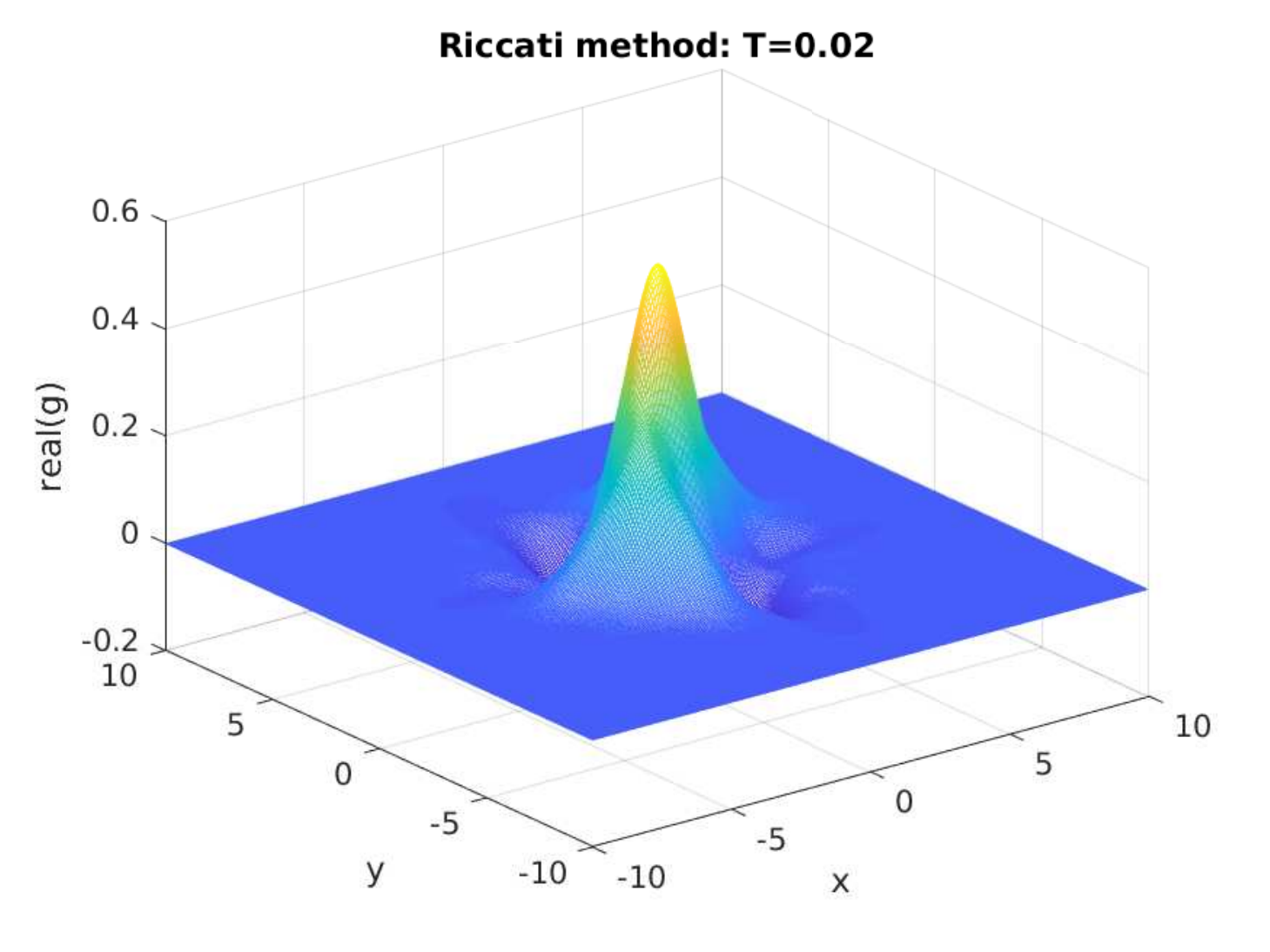}
  \includegraphics[width=5cm,height=5cm]{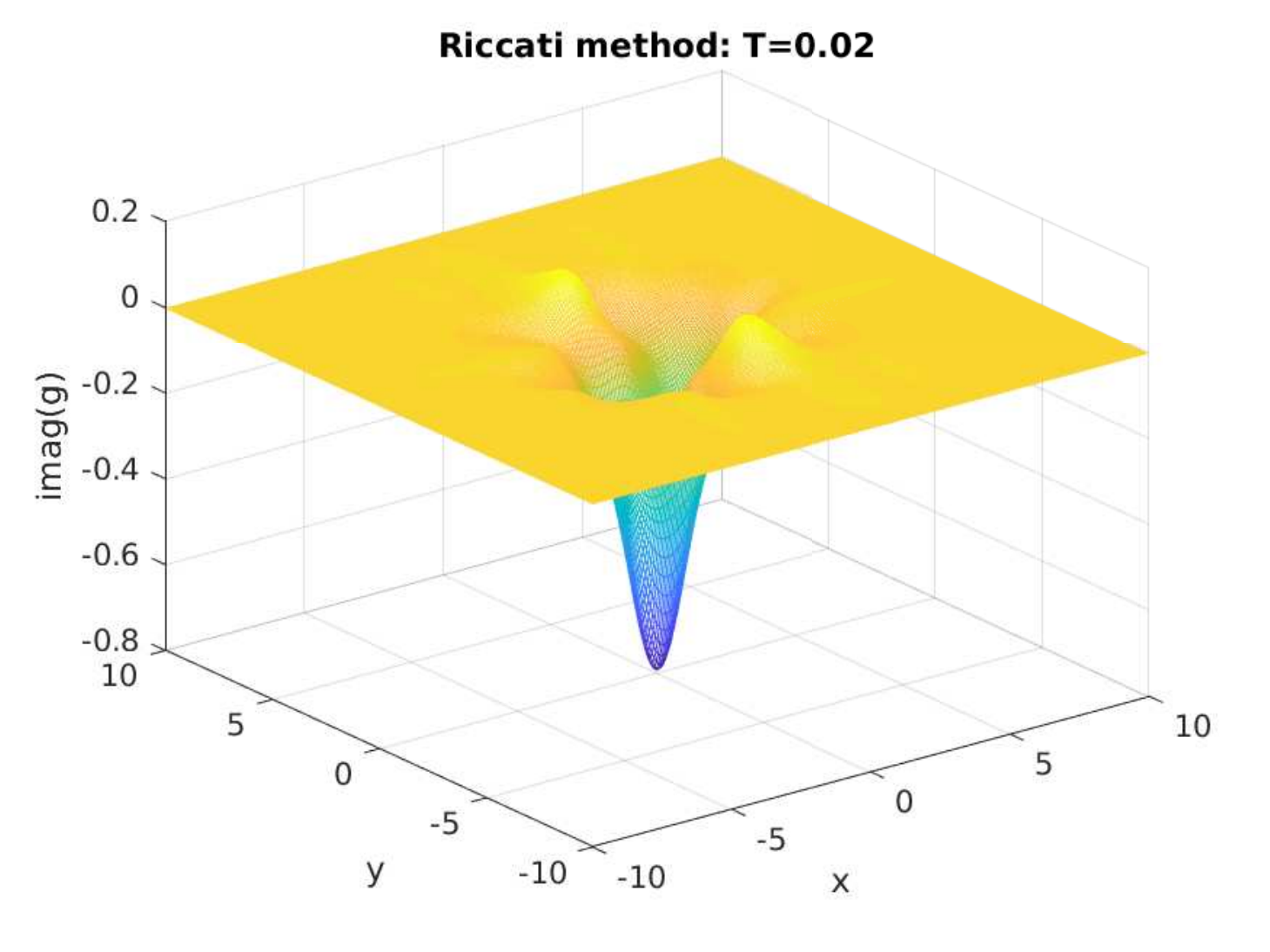}\\
  \includegraphics[width=5cm,height=5cm]{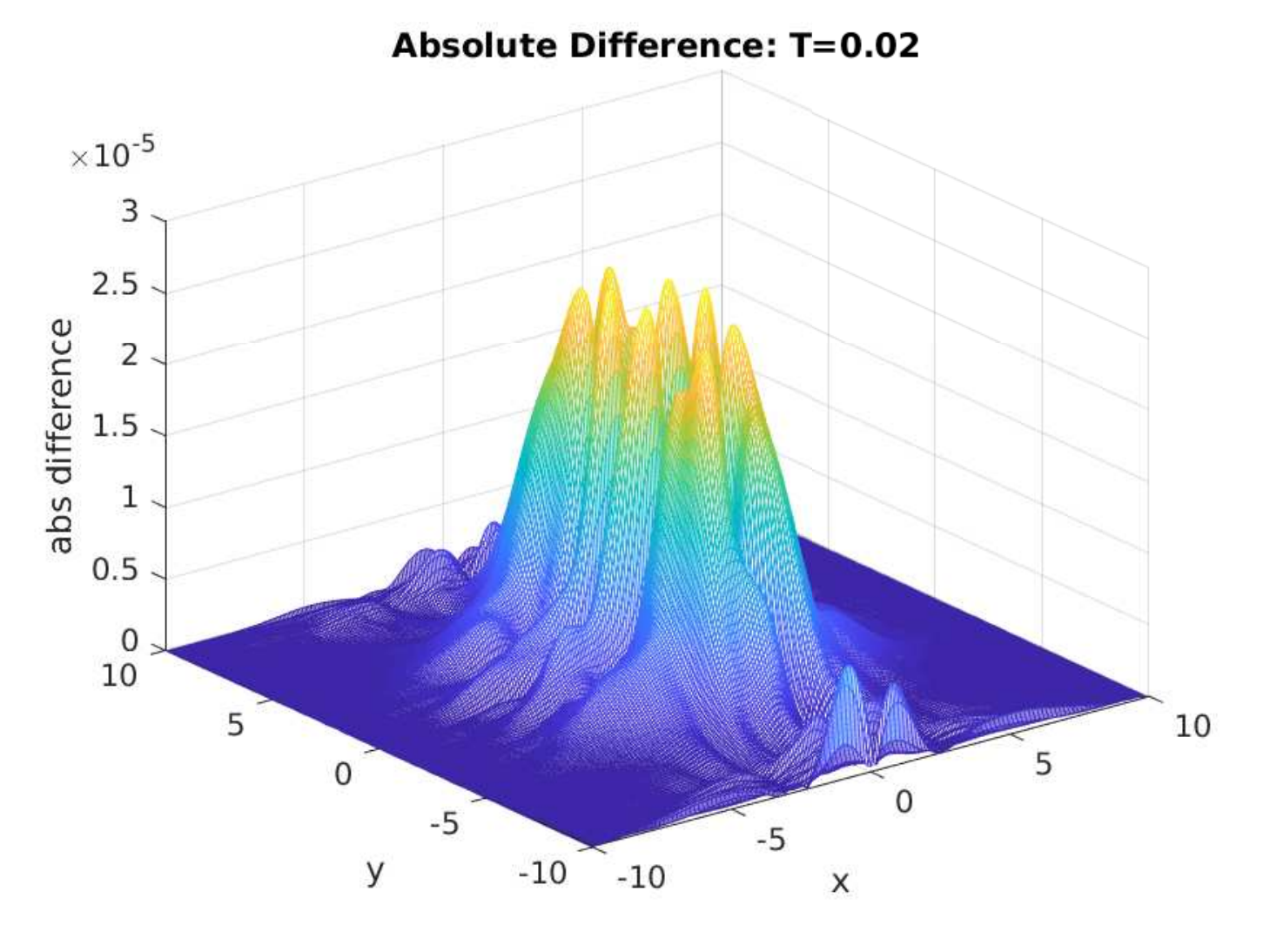}
  \includegraphics[width=5cm,height=5cm]{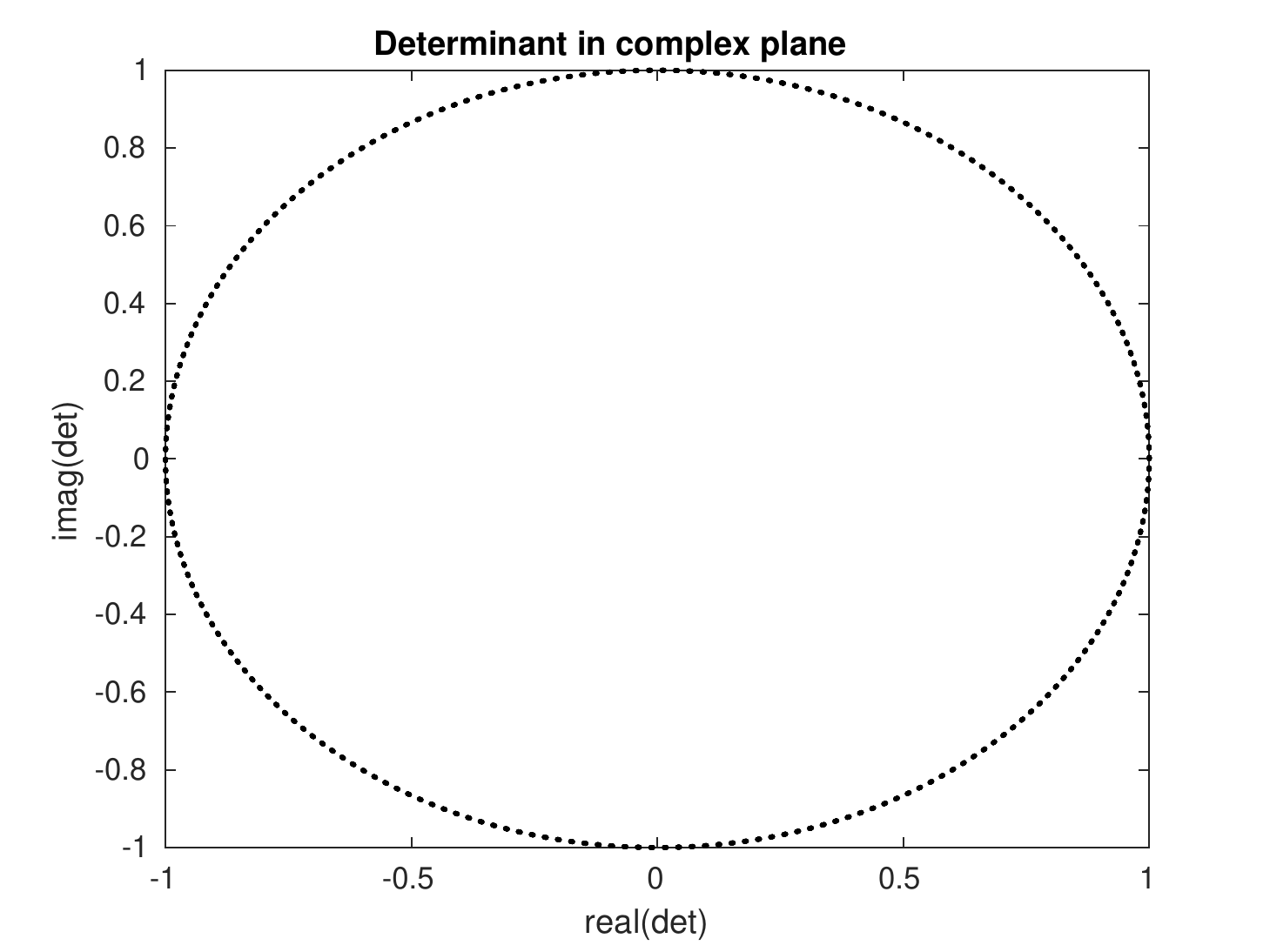}\\
  \end{center}
  \caption{We plot the solution to the cubic nonlocal nonlinear
Sch\"odinger equation from Example~\ref{ex:NLS}. 
We used a generic initial profile 
$g_0(x,y)\coloneqq\mathrm{sech}(x+y)\,\mathrm{sech}(y)$.
For time $T=0.02$, the top panels show the real and imaginary parts 
of the solution computed using a direct integration approach while the middle
panels show the corresponding real and imaginary parts 
of the solution computed using our Riccati approach. The bottom left 
panel show the magnitude of the difference between the 
two computed solutions. The bottom right panel shows the evolution
of the Fourier transform of the Fredholm determinant associated
with $\widehat{Q}(t)$ for $t\in[0,T]$ in the complex plane.}
\label{fig:NLS}
\end{figure}

\begin{example}[Fourth order NLS with nonlocal sinusoidal nonlinearity]\label{ex:genNLS}
In this case the target equation is the fourth order nonlocal nonlinear Sch\"odinger equation
\begin{equation*}
\mathrm{i}\pa_tg=\pa_1^4\,g+g\star\sin^\star\bigl(g\star g^\dag\bigr),
\end{equation*}
for $g=g(x,y;t)$. In our analysis in \S\ref{sec:generalflows} 
we thus need to set $h(x)=x^4$ and $f(x)=\sin(x)$.
The initial profile is $g_0(x,y)\coloneqq \mathrm{sech}(x+y)\,\mathrm{sech}(y)$,
as previously and in this case $L=20$ and $M=2^8$.
The results are shown in Figure~\ref{fig:genNLS}. The top two panels
show the real and imaginary parts of the solution $g_{\mathrm{D}}$ computed up until
time $T=0.2$ using a direct integration approach. By this we mean we implemented the
split-step Fourier transform approach as in the last example, slightly modified
to deal with the sinusoidal nonlinearity, and with time step $\Delta t=0.0001$. 
The middle two panels show the real and imaginary parts of the 
solution $g_{\mathrm{R}}$ computed using our Riccati approach. 
Again by this we mean, given the explicit solution for
$\widehat{p}=\widehat{p}(k,\kappa;t)$ in terms of $\widehat{g}_0$, 
we numerically evaluated $\widehat{q}=\widehat{q}(k,\kappa;t)$
using the exponential form from Lemma~\ref{lemma:explicitlinearodddegree},
now including the sinusoidal form for $f$. We then solved for 
$\widehat{g}=\widehat{g}(k,\kappa;t)$ and so forth, as described in
the last example. The bottom left panel shows 
$|g_{\mathrm{D}}-g_{\mathrm{R}}|$ for all $(x,y)\in[-L/2,L/2]^2$ at time $t=T$. 
Thus again, up to computation error, the solutions naturally coincide
with $\|g_{\mathrm{D}}-g_{\mathrm{R}}\|_{L^\infty(\R^2;\C)}=5.2793\times 10^{-6}$.
As in the last example, the bottom right panel shows the evolution of  
$\mathrm{det}_2\bigl(\widehat{Q}(t)\bigr)$ for $t\in[0,T]$.
Again we observe that $|\mathrm{det}_2\bigl(\widehat{Q}(t)\bigr)|=1$ 
for all $t\in[0,T]$. 
\end{example}

\begin{figure}
  \begin{center}
  \includegraphics[width=5cm,height=5cm]{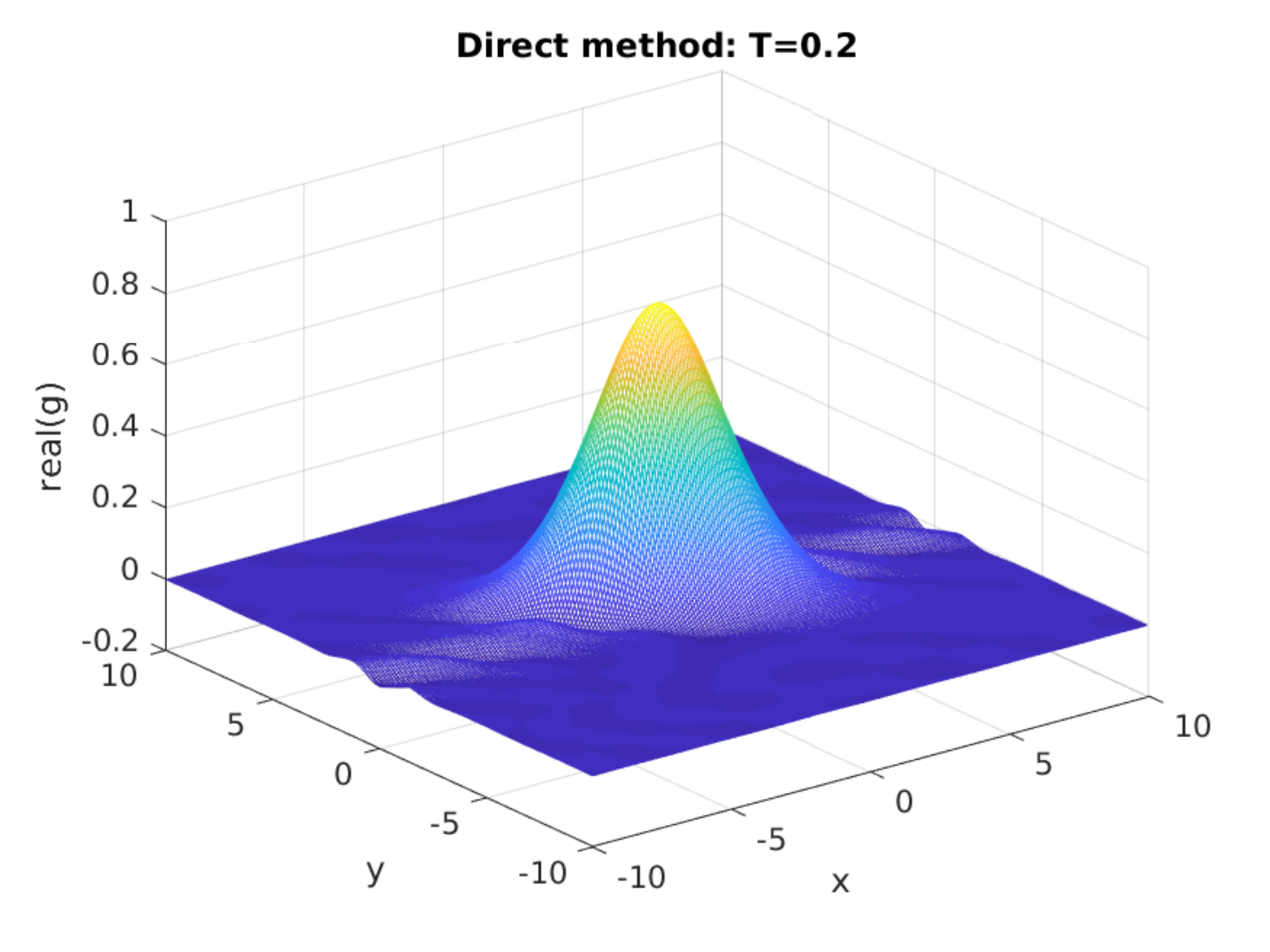}
  \includegraphics[width=5cm,height=5cm]{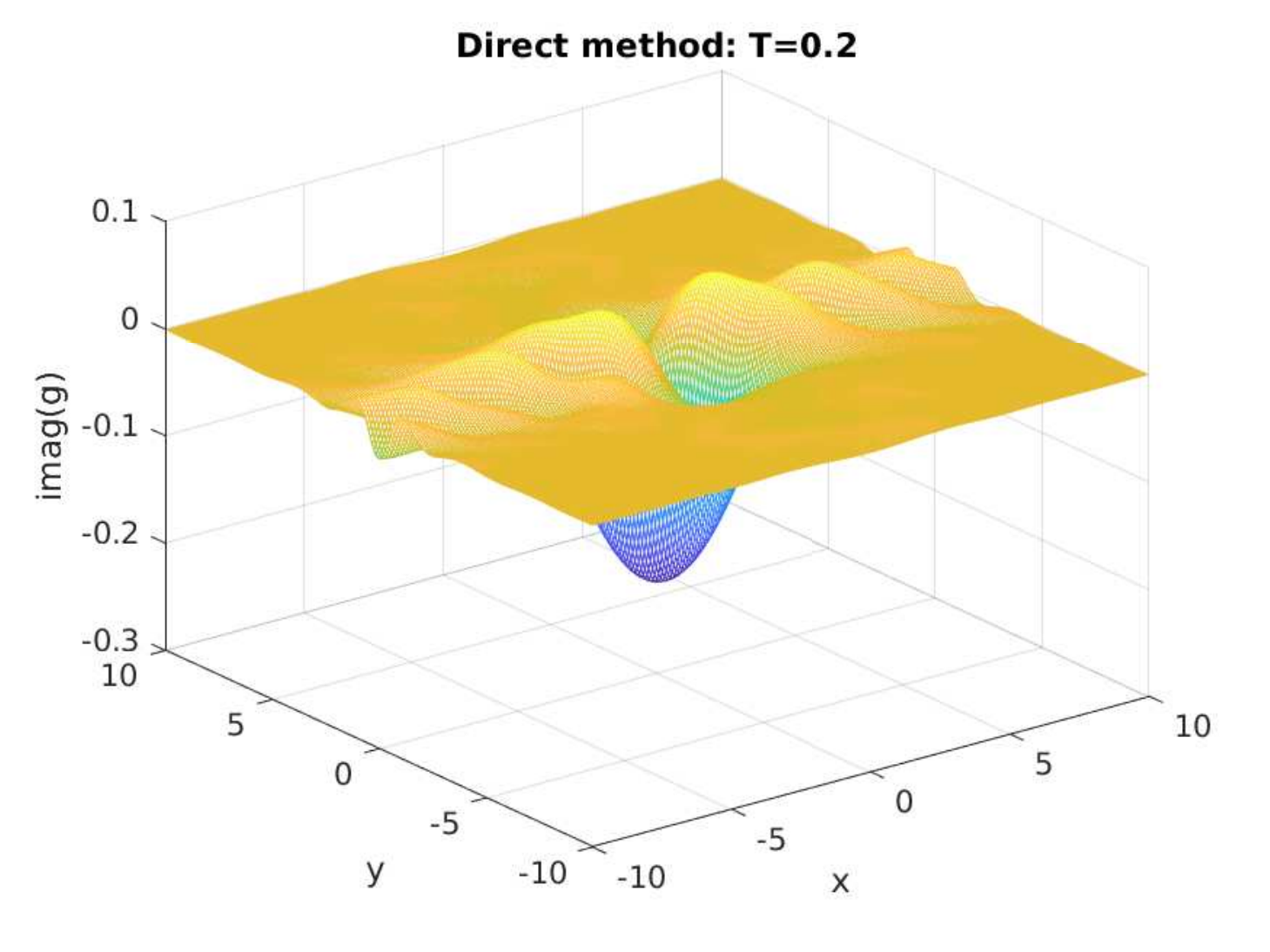}\\
  \includegraphics[width=5cm,height=5cm]{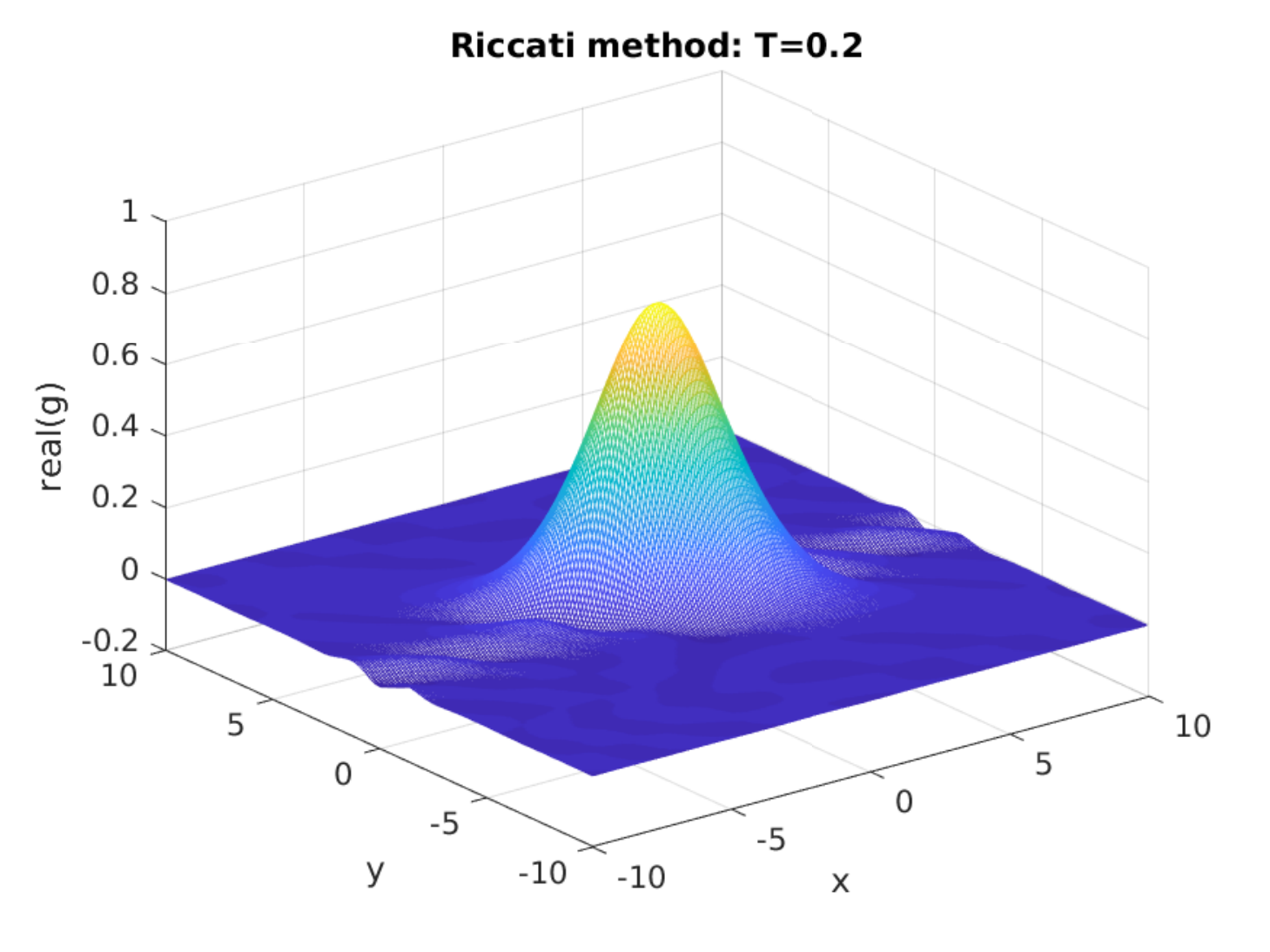}
  \includegraphics[width=5cm,height=5cm]{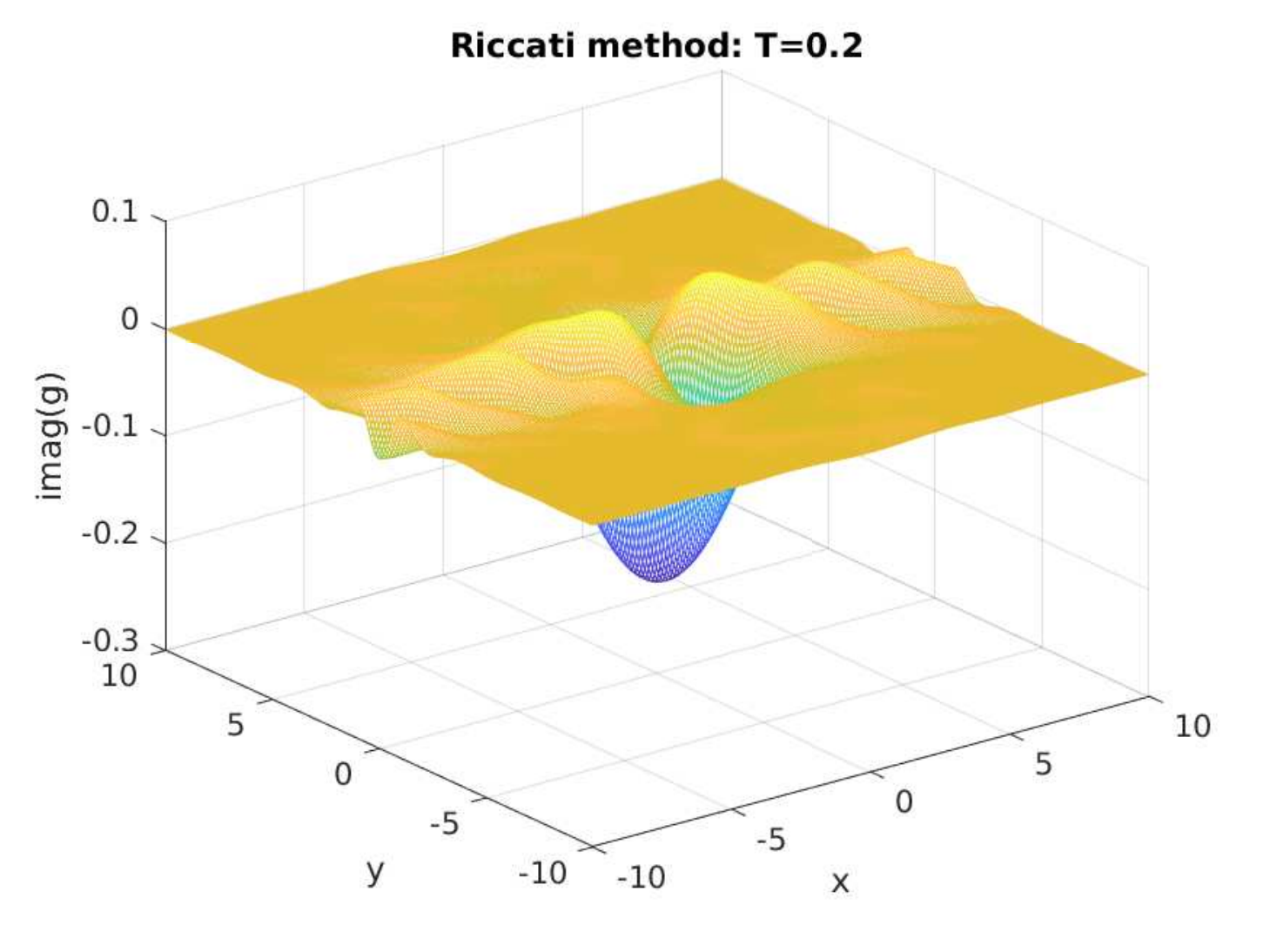}\\
  \includegraphics[width=5cm,height=5cm]{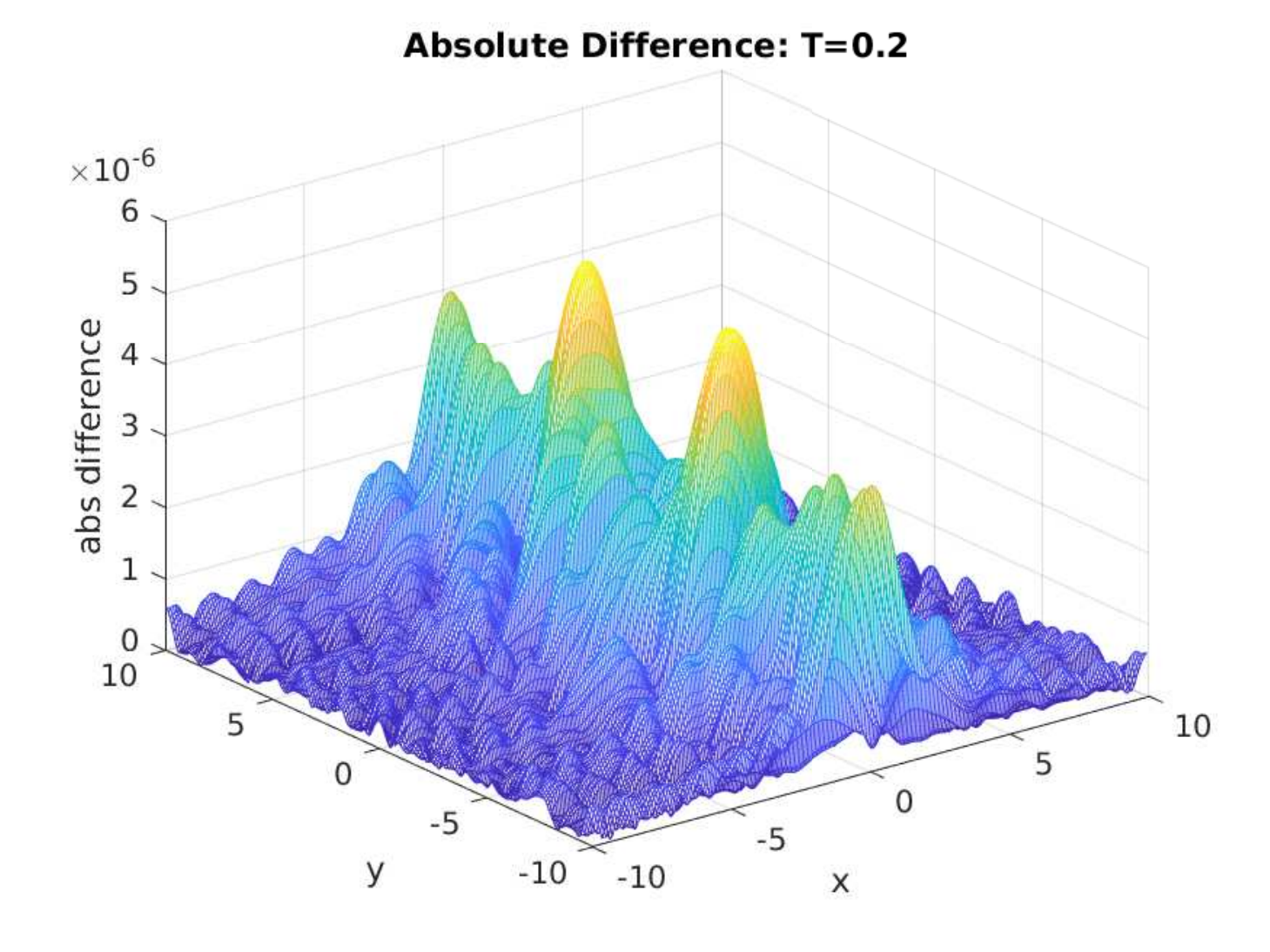}
  \includegraphics[width=5cm,height=5cm]{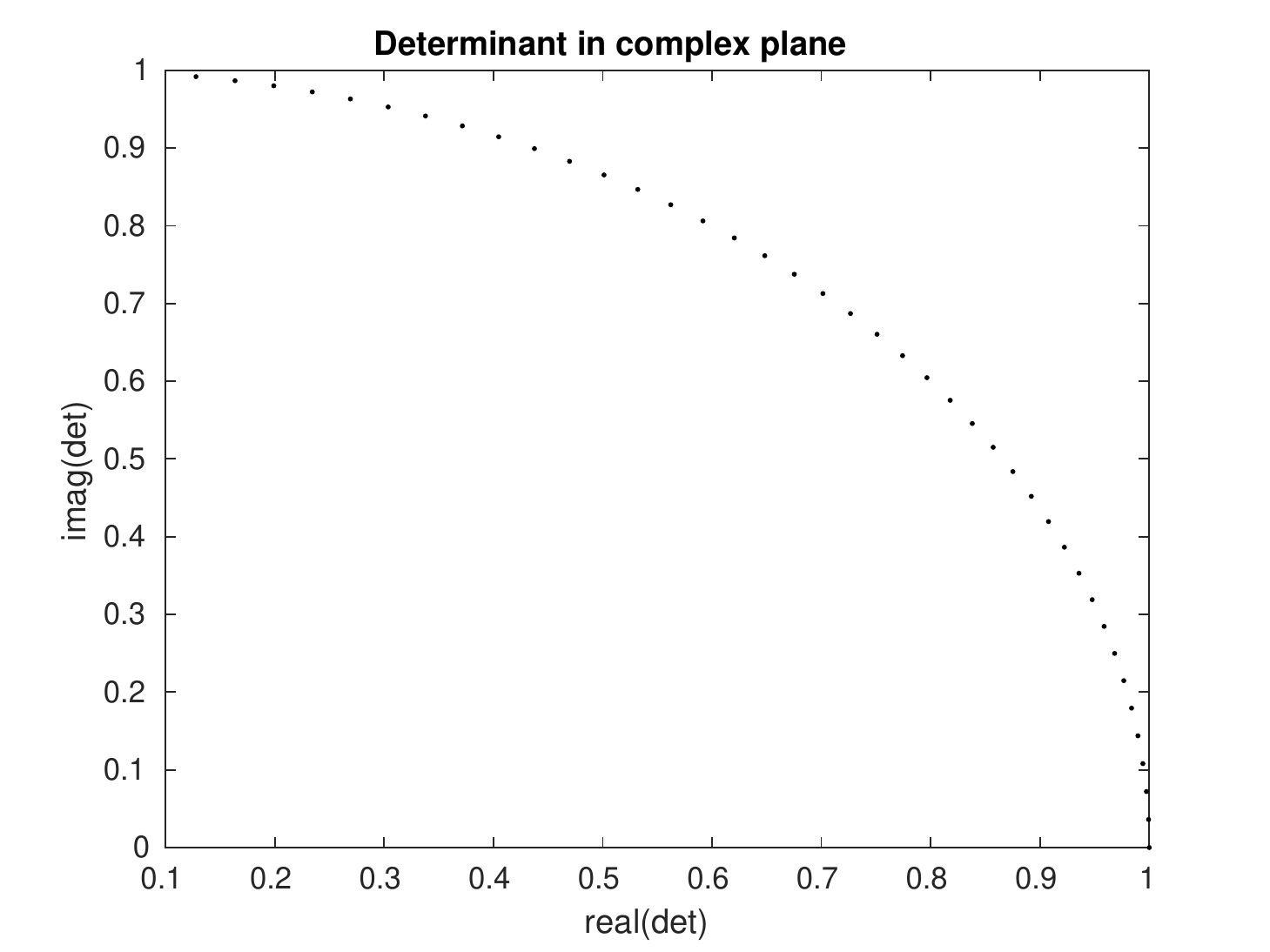}\\
  \end{center}
  \caption{We plot the solution to the nonlocal nonlinear
Sch\"odinger equation with a sinusoidal nonlinearity from Example~\ref{ex:genNLS}. 
We used a generic initial profile 
$g_0(x,y)\coloneqq\mathrm{sech}(x+y)\,\mathrm{sech}(y)$.
For time $T=0.2$, the top panels show the real and imaginary parts 
of the solution computed using a direct integration approach while the middle
panels show the corresponding real and imaginary parts 
of the solution computed using our Riccati approach. The bottom left 
panel shows the magnitude of the difference between the 
two computed solutions. The bottom right panel shows the evolution
of the Fourier transform of the Fredholm determinant associated
with $\widehat{Q}(t)$ for $t\in[0,T]$ in the complex plane.}
\label{fig:genNLS}
\end{figure}

We now present the two special case examples. The first 
is a very special case of the systems in \S\ref{sec:quadflows}
for which the subspace $\Qb$ has co-dimension one with respect
to $\Hb$. We can think of the operator $P$ being parameterized
by an infinite row vector. The second is another special
case when the Riccati relation represents a rank-one
transformation from $Q$ to $P$. Here we use this context to solve
a particular version of the nonlocal 
Fisher--Kolmogorov--Petrovskii--Piskunov equation. 
The Cole--Hopf transformation for the Burgers equation also
represents such a rank-one case; see Beck \textit{et al.\/} \cite{BDMS}.

\begin{example}[Evolutionary diffusive PDE with convolutional nonlinearity]
\label{ex:conv}
In this example we assume the linear base and auxiliary equations 
have the form
\begin{equation*}
\pa_tp(y;t)=d(\pa_y)\,p(y;t)
\qquad\text{and}\qquad
\pa_tq(y;t)=b(\pa_y)\,p(y;t).
\end{equation*}
In these equations we assume the operator $d=d(\pa_y)$
is a polynomial in $\pa_y$ with constant coefficients and that it is 
of diffusive or dispersive type as described in \S\ref{sec:quadflows}. 
We also assume $b=b(\pa_y)$ is a polynomial in $\pa_y$ 
with constant coefficients. We now posit the Riccati relation 
\begin{equation*}
p(y;t)=\int_\R g(z;t)\,q(z+y;t)\,\rd z.
\end{equation*}
Following Remark~\ref{remark:directview} in
\S\ref{sec:quadflows} by differentiating this Riccati relation
with respect to time and using that $p=p(y;t)$ and
$q=q(y;t)$ satisfy the scalar linear base and auxiliary equations above,
we find 
\begin{align*}
\int_\R \pa_tg(z;t)\,q(z+y;t)\,\rd z
=&\;\pa_tp(y;t)-\int_\R g(z;t)\,\pa_tq(z+y;t)\,\rd z\\
=&\;d(\pa_y)\,p(y;t)-\int_\R g(z;t)\,b(\pa_z)\,p(z+y;t)\,\rd z\\
=&\;\int_\R g(z;t)\,d(\pa_y)\,q(z+y;t)\,\rd z\\
&\;-\int_\R g(z;t)\,b(\pa_z)\,\int_\R g(\zeta;t)\,q(\zeta+z+y;t)\,\rd\zeta\,\rd z\\
=&\;\int_\R \bigl(d(-\pa_z)\,g(z;t)\bigr)\,q(z+y;t)\,\rd z\\
&\;-\int_\R \bigl(b(-\pa_z)\,g(z;t)\bigr)\,\int_\R g(\zeta;t)\,q(\zeta+z+y;t)\,\rd\zeta\,\rd z\\
=&\;\int_\R \bigl(d(-\pa_z)\,g(z;t)\bigr)\,q(z+y;t)\,\rd z\\
&\;-\int_\R \bigl(b(-\pa_z)\,g(z;t)\bigr)\,\int_\R g(\xi-z;t)\,q(\xi+y;t)\,\rd\xi\,\rd z\\
=&\;\int_\R \bigl(d(-\pa_z)\,g(z;t)\bigr)\,q(z+y;t)\,\rd z\\
&\;-\int_\R\int_\R \bigl(b(-\pa_\xi)\,g(\xi;t)\bigr)\,g(z-\xi;t)\,\rd\xi\, q(z+y;t)\,\rd z.
\end{align*}
Here we integrated by parts assuming suitable decay in the far-field,
used the substitution $\xi=\zeta+z$ for fixed $z$, and swapped over the 
integration variables $\xi$ and $z$. As in Remark~\ref{remark:directview},
if we postmultiply by `$\delta(z-y)+\tilde{q}^\prime(y,\eta;t)$' and integrate
over $y\in\R$ we find $g=g(\eta;t)$ satisfies 
\begin{equation*}
\pa_tg(\eta;t)=d(-\pa_\eta)\,g(\eta;t)
-\int_\R \bigl(b(-\pa_\xi)\,g(\xi;t)\bigr)\,g(\eta-\xi;t)\,\rd\xi.
\end{equation*}
This is a simpler derivation of Example~1 from 
Beck \textit{et al.\/}~\cite{BDMS} where $b=1$. 
There we derive an explicit form for the Fourier transform of
the solution and compare the result of direct numerical simulations
with evaluation of the solution using our explicit formula.
\end{example}

\begin{example}[Nonlocal Fisher--Kolmogorov--Petrovskii--Piskunov equation]
\label{ex:nonlocalFKPP}
In this example we assume the scalar linear base and auxiliary equations 
have the form
\begin{equation*}
\pa_tp(x;t)=d(\pa_x)\,p(x;t)
\qquad\text{and}\qquad
\pa_tq(x;t)=b(x,\pa_x)\,p(x;t).
\end{equation*}
Here the operator $d=d(\pa_x)$ is assumed to be a polynomial
in $\pa_x$ with constant coefficients of diffusive or dispersive
type as described in \S\ref{sec:quadflows}. We assume that the
operator $b=b(x,\pa_x)$ is either of the form $b=b(x)$ only, where
$b(x)$ is a bounded function, or it is of the form $b=b(\pa_x)$ only,
in which case we assume it is a polynomial in $\pa_x$ with constant 
coefficients. We could assume $b=b(x,\pa_x)$ is a polynomial in $\pa_x$
with non-homogeneous coefficients, the main constraint is whether
we can find an explicit form for the solution $q=q(x;t)$ to the linear
auxiliary equation. We now posit the Riccati relation of the following
rank-one form
\begin{equation*}
p(x;t)=g(x;t)\,\int_\R q(z;t)\,\rd z.
\end{equation*}
For convenience we set 
$\overline{q}(t)\coloneqq\int_\R q(z;t)\,\rd z$,
in which case we have $p(x;t)=g(x;t)\,\overline{q}(t)$ and 
\begin{equation*}
\pa_t\overline{q}(t)=\int_\R b(z,\pa_z)\,p(z;t)\,\rd z.
\end{equation*}
As in \S\ref{sec:quadflows}, in particular for example 
in Remark~\ref{remark:directview}, we differentiate the
Riccati relation with respect to time and substitute in
that $p=p(x;t)$ satisfies the linear base equation and 
$\overline{q}=\overline{q}(t)$ satisfies the equation
just above. Carrying this through generates
\begin{align*}
\bigl(\pa_tg(x;t)\bigr)\,\overline{q}(t)
&=\pa_tp(x;t)-g(x;t)\,\pa_t\overline{q}(t)\\
&=d(\pa_x)p(x;t)-g(x;t)\,\int_\R b(z,\pa_z)\,p(z;t)\,\rd z\\
&=d(\pa_x)g(x;t)\,\overline{q}(t)
-g(x;t)\,\int_\R b(z,\pa_z)\,g(z;t)\,\rd z\,\overline{q}(t).
\end{align*}
Dividing through by $\overline{q}=\overline{q}(t)$ generates the equation
\begin{equation*}
\pa_tg(x;t)=d(\pa_x)g(x;t)-g(x;t)\,\int_\R b(z,\pa_z)\,g(z;t)\,\rd z.
\end{equation*}
Now suppose we wish to solve this evolutionary partial differential
equation with the nonlocal nonlinearity shown for some given initial data $g_0(x)$,
i.e.\/ such that $g(x;0)=g_0(x)$. We naturally take $\overline{q}(0)=1$ 
and $p(x;0)=g_0(x)$. Then that $g(x;t)=p(x;t)/\overline{q}(t)$ is indeed 
the corresponding solution to the evolutionary partial differential equation
for $g=g(x;t)$ above, with $p=p(x;t)$ satisfying the linear base equation
above and $\overline{q}=\overline{q}(t)$ satisfying the integrated
auxiliary equation shown, can be verified by direct substitution.

Let us now consider the special case $b=1$. Then 
by analogy with Lemma~\ref{lemma:explicitlinearquadratic}, the 
solution $p=p(x;t)$ to the linear base equation is given in terms of its Fourier
transform by 
\begin{equation*}
\widehat{p}(k;t)=\exp\bigl(d(2\pi\mathrm{i}k)\,t\bigr)\,\widehat{g}_0(k).
\end{equation*}
By taking the inverse Fourier transform of this and integrating with
respect to the spatial coordinate, we find the solution $\overline{q}=\overline{q}(t)$ 
to the integrated auxiliary equation is then given by
\begin{equation*}
\overline{q}(t)=1+\biggl(\frac{\exp(t\,d(0))-1}{d(0)}\biggr)\,\widehat{g}_0(0).
\end{equation*}
If $d(0)=0$, this becomes $\overline{q}(t)=1+t\,\widehat{g}_0(0)$.
Hence we have an explicit solution for any diffusive or
dispersive form for $d=d(\pa_x)$. If $d(\pa_x)=\pa_x^2+1$,
the partial differential equation for $g=g(x;t)$ above
corresponds to a particular version of the nonlocal 
Fisher--Kolmogorov--Petrovskii--Piskunov equation which is studied 
for example in Britton~\cite{Britton} and Bian, Chen \& Latos~\cite{BCL}. 
\end{example}

\section{Conclusion}\label{sec:conclu}
We have extended our Riccati approach for generating solutions to
nonlocal nonlinear partial differential equations from a corresponding
linear base equation to systems as well as higher odd degree nonlinearities.
These systems can be of arbitrary order in the linear terms and include
higher order terms in the nonlocal nonlinear terms.
We also provided explicit calculations demonstrating how solutions for 
such nonlocal nonlinear systems can be generated in this manner for 
general initial data. For four example systems we also provided numerical
simulations comparing solutions computed using the Riccati approach 
and solutions computed using direct primarily pseudo-spectral 
numerical methods. We provide all the Matlab codes in the supplementary
electronic material. We also indicated multiple immediate extensions 
we intend to consider, for example to tackle the case of 
higher even degree nonlinearities. Additionally we hinted on how
we intend to extend the Riccati approach to the multi-dimensional 
nonlocal nonlinear partial differential equations.

There are many further extensions and practical considerations in our sights. 
One natural extension is to consider using the Riccati approach for nonlocal 
nonlinear stochastic partial differential equations. We would begin with those
with additive space-time noise which could be incorporated via the operator $C$
in the quadratic nonlocal nonlinearity set-up described in \S\ref{sec:quadflows}.
It appears as a linear term in the base equation which would thus become
a linear stochastic partial differential equation. The base and auxiliary
equations would have to be solved as a linear system, which is achievable
in principle. Then the term $C$ appears as a nonhomogeneous source term in
the final Riccati stochastic partial differential equation. Indeed 
we have already performed some simulations of this nature and these
will be published in Doikou, Malham and Wiese~\cite{DMW}.
On the practical consideration side, we note that to compute solutions
using the Riccati method in practice, we may need to approximate the
solution to the linear auxiliary equation, and then typically, we need to solve 
the linear Fredholm integral equation numerically to find the desired solution. 
It would be useful to provide a comprehensive numerical analysis study examining 
the relative complexity of the Riccati approach in these cases compared to the 
state-of-the-art numerical methods available for such nonlinear systems.

The context and examples we have considered thus far have included large
classes of nonlocal nonlinear systems. One way to classify these systems 
is that they can all be thought of as "big matrix" equations with the 
natural extended product encoded in the `$\star$' product. In other words
we think of the linear operators $P$, $Q$ and $G$ as matrix operators 
extended to the infinite-dimensional context, whether countable or not. 
The resulting objects are either countably infinite matrices or are parametrized
by integral kernels. The natural extension of the matrix product is then
the countable discrete version of the star product or the star product
itself. One of our next goals is to consider how to generalize our Riccati approach
so as to incorporate local nonlinearities. One natural approach is to replace 
the Fredholm Riccati relation by a Volterra one.

Lastly, the classes of nonlinear partial differential equations we have
considered may have solutions which become singular in finite time. 
For example the nonlocal nonlinear Schr\"odinger equation with higher
degree nonlinearity or in higher dimensions might exhibit such behaviour.
However let us consider the overarching context of the Riccati approach 
we prescribe which is that of a linear subspace flow projected down onto the 
Fredholm Grassmannian. In principle the solutions to the 
underlying linear base and auxiliary equations which generate 
the solution to the nonlocal nonlinear system do not themselves
become singular in finite time. The singularity in the 
nonlocal nonlinear system is just an artifact of a poor 
choice of coordinate patch on the Fredholm Grassmannian.
It corresponds to the event $\mathrm{det}_2\bigl(\id+Q^\prime(t)\bigr)\to0$,
though we need to be wary of a hierarchy of regularized determinants here that
should be monitored. The coordinate patch choice is made in the projection
\begin{equation*}
\begin{pmatrix} Q\\P\end{pmatrix}\to\begin{pmatrix} \id\\G\end{pmatrix}.
\end{equation*}
Implicit in the projection as shown is that we have equivalenced by
the ``top'' block of suitable general linear transformations, thus generating 
the graph and coordinate patch on the right shown. However we can
equivalence by \emph{any} block of suitable general linear transformations
(for example the lower block instead) generating a different graph
and coordinate patch. Indeed there is a Schubert cell decomposition
of the Fredholm Grassmannian analogous to that in the finite-dimensional case;
see Pressley and Segal~\cite{PS}. Careful analysis of the behaviour of 
the solutions to the underlying linear base and auxiliary equations
on the approach to and transcending through and beyond the singularity
in a given coordinate patch might reveal more detailed information about
the singularity and will provide a mechanism for continuing solutions
beyond it. 


\begin{acknowledgement}
We would like to thank the referees for their insightful comments and constructive
suggestions that helped significantly improve the original manuscript.
We would also like to thank Anke Wiese for her helpful comments and suggestions and 
Jonathan Sherratt for useful discussions. The work of M.B. was partially supported by 
US National Science Foundation grant DMS-1411460.
\end{acknowledgement}

\end{document}